\documentclass[a4paper,10pt]{article}

\usepackage{amsfonts}
\usepackage{amssymb}
\usepackage{amsthm}
\usepackage{amsmath}
\usepackage{a4wide}
\usepackage{mathrsfs}
\usepackage{epsfig}
\usepackage{palatino}
\usepackage{tikz,fp,ifthen}
\usepackage{float}
\usepackage{lipsum}

\newcommand{\pdfgraphics}{\ifpdf\DeclareGraphicsExtensions{.pdf,.jpg}\else\fi}
\usepackage{graphicx}

\usepackage{color}
\definecolor{hanblue}{rgb}{0.27, 0.42, 0.81}
\definecolor{red}{rgb}{1.0, 0.0, 0.0}
\usepackage[colorlinks, citecolor=hanblue,linkcolor=blue, urlcolor = hanblue,hypertexnames=false]{hyperref}

\usepackage{mathtools}
\mathtoolsset{showonlyrefs}

\numberwithin{equation}{section}

\theoremstyle{plain}

\newtheorem{thm}{Theorem}[section]
\newtheorem{lem}[thm]{Lemma}
\newtheorem{prop}[thm]{Proposition}
\newtheorem{cor}[thm]{Corollary}

\theoremstyle{definition}
\newtheorem{defn}[thm]{Definition}

\theoremstyle{remark}

\numberwithin{equation}{section}

\def\eps{\varepsilon}

\def\loc{_{\operatorname{loc}}}
\def\NN{\mathbb N}

\def\R{\mathbb R}

\renewcommand{\t }{\tau }

\newcommand{\intbar}{\etaathop{\int\etaakebox(-13.5,0){\rule[4pt]{.7em}{0.3pt}}
\kern-6pt}\nolimits}

\newcommand{\be}{\begin{equation}}
\newcommand{\ee}{\end{equation}}
\newcommand{\bea}{\begin{equation*}}
\newcommand{\eea}{\end{equation*}}

\def\loc{_{\operatorname{\rm loc}}}

\def\R{{{\mathbb R}}}
\def\SS{{{\mathbb S}}}

\def\NN{{{\mathbb N}}}

\def\eps{\varepsilon}

\def\dert{\partial_t}
\def\ders{\partial_s}

\def\be{\begin{equation}}
\def\ee{\end{equation}}
\def\bea{\begin{eqnarray*}}
\def\bean{\begin{eqnarray}}
\def\eean{\end{eqnarray}}
\def\eea{\end{eqnarray*}}

\begin{document}
\pdfgraphics 

\title{\mbox{Type-0 singularities in the network flow -- Evolution of trees}}

\author{Carlo Mantegazza \footnote{Dipartimento di Matematica e Applicazioni, Universit\`a di Napoli Federico II, Via Cintia, Monte S. Angelo
80126 Napoli, Italy} \and Matteo Novaga \footnote{Dipartimento di Matematica, Universit\`a di Pisa, Largo
    Bruno Pontecorvo 5, 56127 Pisa, Italy} \and Alessandra Pluda\footnotemark[2]}
\date{}

\maketitle

\begin{abstract}
\noindent
The motion by curvature of networks is the generalization to
finite union of  curves of the curve shortening flow. This evolution has
several peculiar features, mainly due to the presence of junctions where the curves
meet.
In this paper we show that whenever the length of one single curve
vanishes and two triple junctions coalesce, then the curvature of the evolving networks
remains bounded. This topological singularity is exclusive of the network flow
and it can be referred as a ``Type-0'' singularity, in contrast to the well known 
``Type-I'' and ``Type-II'' ones of the usual mean curvature flow of smooth curves or hypersurfaces, characterized
by the different rates of blow up of the curvature.
As a consequence, we are able to give a complete description of the 
evolution of tree--like networks till the first singular time,
under the assumption that all the ``tangents flows'' have unit multiplicity.
If the lifespan of such solutions is finite, then the curvature of the network remains bounded and we can apply the results from~\cite{IlNevSchu,LiMaPluSa} to ``restart'' the flow after the singularity. 
\end{abstract}

%


\section{Introduction}

One of the most studied geometric evolution equations is the {\em mean curvature flow}: a family of time--dependent $n$--dimensional hypersurfaces in $\mathbb{R}^{n+1}$ that evolve with normal velocity equal to the mean curvature at any point. Such flow can be interpreted as the gradient flow of the area functional. A natural generalization is the so--called 
{\em multiphase mean curvature flow}, where more than one hypersurface evolves in time by locally moving in the direction of the mean curvature. The presence of more than one hypersurface creates extra difficulties in dealing with the problem,
since the evolving objects are not smooth manifolds anymore
and one will have coupling conditions among the evolving surfaces due to their interactions.
Anyway, one of the earliest motivations for the study of the mean curvature flow was to model 
the grain growth (the increase in size of grains/crystallites in a material at high temperature), where considering more than one evolving surfaces is actually necessary.

\medskip

The one--dimensional version of such multiphase mean curvature flow is usually called
{\em motion by curvature of networks} or {\em network flow}, that is, a network of curves moving in the plane decreasing the total length in the ``most efficient'' way and hopefully converging to a configuration of minimal length (or at least a ``critical'' configuration for the ``total length functional'' on the networks).

There are basically two approaches to this problem, related to considering weak or strong (smooth) solutions, even if they have several points of contact. An advantage of the weak formulations is that one gets easily global existence results, but then showing uniqueness and regularity becomes quite hard. On the other side, well posedness of strong solutions is easier, but it is then complicated to define a flow passing past the singularities.

\medskip

Different notions of weak solutions have been investigated by several authors,
for instance~\cite{altawa,brakke,es,KiTo,LaOt}.
A very relevant one is the weak definition of the flow given in the seminal work by Brakke~\cite{brakke}, that is, a family of $n$--dimensional varifolds $\{V_t\}_{t\in [0,T)}$ such that  for 
almost every time, $V_t$ is an integral $n$--varifold with locally bounded first variation and ``generalized mean curvature'' in $L^2$, satisfying a suitable distributional inequality.
Recent progresses concerning the regularity of these solutions
has been made by Kim and Tonegawa~\cite{KiTo,KiTo2}
at least for codimension--one varifolds, ``coupling'' the evolving varifolds with 
a finite number of time--dependent mutually disjoint open sets which are the regions bounded by such varifolds. 
In this setting, when the initial datum is a closed 1--rectifiable set in $\mathbb{R}^2$ with (locally) finite measure, then for almost every time  the support of the evolving varifolds consists of embedded $W^{2,2}$--curves whose end--points meet at junctions forming angles of $0$, $60$ or $120$ degrees. 
Another possible way to define weak solutions is the so--called ``minimizing movements'' scheme, first introduced in~\cite{altawa,LuSt} and later extended to the network case in~\cite{BeKh,degio4}, which consists in defining the flow as limit of time-discrete flows, defined variationally by minimizing a penalized geometric energy. Finally, we mention that one can also consider a phase--field approximation of the flow given by the singular limits of solutions of a parabolic vector--valued Allen--Cahn equation, see~\cite{AlCh,Chen,LaSi}.

\medskip

The long term goal of our project is a complete description of the network flow 
defined in the framework of classical PDEs, that is, we are interested in smooth solutions far from singular times.
We consider the evolution of a finite union of smooth curves meeting at multiple junctions, since the network flow is the gradient flow of the length, it is expected (by ``energetic'' reasons) that for almost all times the evolving networks present
only triple junctions where the unit tangent vectors of the three 
concurring curves sum up to zero.
Hence, it is natural to first to restrict the analysis to the class of ``regular'' networks,
where curves meet only at triple junctions forming angles of $120$ degrees.
Then, the motion can be written as a system of quasilinear second order parabolic equations
with nonlinear boundary conditions (see~\eqref{motionsystem}).

\medskip

The first well posedness result for this problem has been established in the '90
by Bronsard and Reitich~\cite{BroRei} whenever the initial network (all its curves) is of class
$C^{2+\alpha}$ and the {\em compatibility conditions} for the system of PDEs, up to the second order, are fulfilled at the triple junctions (namely, both the sum of the three unit tangent vectors and the sum of the three curvatures are equal to zero, see~\cite[Definition 4.2]{ManNovPluSchu}). Existence of solutions can actually be achieved when the initial datum is only of class $C^2$,
without requiring second order conditions at the boundary~\cite[Theorem~6.8]{ManNovPluSchu}. 
Since it is then possible to prove existence and uniqueness 
when the initial network belongs to the fractional Sobolev space $W^{2-\frac{2}{p},p}$ 
with $p\in (3,+\infty)$ (as shown in~\cite[Theorem~1]{GoMePlu}),
we get a fortiori existence and uniqueness for $C^2$ initial data.

\medskip

Since by these results the flow starts, we can then study the long time behavior of the solutions.
By~\cite[Theorem~3.18]{ManNovTor}, if the maximal time of existence of a solution is finite, then
either the curvature blows up or the length of at least one curve goes to zero.
For instance, the area of a region bounded by less then six curve decrease linearly in time, then when the region vanishes, it can be shown that the curvature blows up. It is also easy to find an example where only he length of a single curve goes to zero and either
one triple junction collapses onto one boundary point or two triple junctions coalesce.
However, there are no examples in which the lengths of all the curves 
of the network remain strictly positive and the curvature goes to $+\infty$. Thus, one may guess that if no length goes to zero, then there are no  singularities. It turns out that this conclusion would be a consequence of the following conjecture.\\

\textbf{Multiplicity--One Conjecture (M1)}.
Every limit of parabolic rescalings of the flow around a fixed point in $\R^2$ is a flow of embedded networks with multiplicity one.\\

We are indeed able to prove that it if no length goes to zero and this conjecture is true, then the curvature remains bounded during the flow. In other words, there are no singularities until at least one curve vanishes.

\medskip

As we said before, if a region disappears at the maximal time of existence, then the curvature cannot remain bounded. One may guess that if the initial network does not contain any loop, hence there are no regions that can collapse, this type of singularities are excluded. We actually have a positive answer.

\begin{thm}\label{bddcurvature}
Let $\mathcal{N}_0$ be a regular network which is a {\em tree}. Suppose that
the Multiplicity--One Conjecture is true 
and that $\mathcal{N}_t$
is the maximal solution to the motion by curvature in the time interval 
$[0,T)$, with initial datum $\mathcal{N}_0$.
Then  the curvature of $\mathcal{N}_t$ is uniformly bounded during the flow.
\end{thm}

We briefly outline the proof of this result. 
We consider suitable rescaled flows around points of $\mathbb{R}^2$.
Supposing that every limit of parabolic rescalings has multiplicity one
and that the evolving network is a {\em tree}, then the limit of our rescaling, 
when not empty, can only be the ``static'' flow given by a line, a standard triod
or a standard cross (see Figure~\ref{standard})
as we show in Lemma~\ref{possibiliblowup}.
In the case of the line or of a standard triod, then 
the original flow has bounded curvature thanks to results by White~\cite{Wh},
Magni, Mantegazza and Novaga~\cite{MaManNov} and Ilmanen, Neves and Schulze~\cite{IlNevSchu}.

It remains to show the same conclusion when the limit is a standard cross, and we will show this in Theorem~\ref{cross}. First, in Lemma~\ref{kinfty} and~\ref{kappa2} we prove that for any tree, if we assume a uniform control on the motion of its end--points, the $L^2$--norm of its curvature $\kappa$ is uniformly bounded in a time interval depending on its initial value. Moreover, red in the same hypotheses, we obtain an analogous conclusion on the $L^\infty$--norm of the curvature in terms of its $L^2$--norm and of the $L^2$--norm of its derivative in arclength $\partial_s\kappa$. Then, we prove that for a special tree, composed by only five curves, two triple junctions and four end--points on the boundary of $\Omega\subseteq\R^2$ open, convex and regular (see Figure~\ref{tree}), uniformly controlling, as before, its end--points and the lengths of the ``boundary curves'' from below, the $L^2$--norm of $\partial_s\kappa$ is bounded until $\Vert k\Vert_{L^2}$ stays bounded. By localizing these estimates one finally gets 
Theorem~\ref{cross}, thus Theorem~\ref{bddcurvature} follows.

\medskip

Even if the curvature is then always controlled, in Section~\ref{extypezero} we show that 
singularities might still occur during the evolution, due to topological changes when the length of a curve goes to zero.
Indeed, we explicitly construct an example of a network
composed of five curves where two triple junctions coalesce in finite time
with the vanishing of one curve.
Because of their peculiar nature, we refer to these singularities with bounded curvature as {\em Type-0} singularities (suggestion by Tom~Ilmanen). We stress the fact that they are exclusive of the network flow, as they cannot appear in the motion by curvature of a single curve. 

\medskip

It is worth mentioning that, if a network contains at most two triple junctions,
we can prove that every limit of parabolic rescalings is a flow of embedded networks with multiplicity one~\cite{ManNovPlu}, but unfortunately the Multiplicity--One Conjecture is still an open problem in full generality.
Assuming it, we have a quite complete, ``nice'' description of the flow of tree--like networks till the first singular time.

\begin{thm}\label{main}
Let $\mathcal{N}_0$ be a regular tree composed of $N$ curves in an open, convex and regular $\Omega\subseteq\R^2$ Suppose that $\mathcal{N}_t$ is the maximal solution to the motion by curvature in the time interval 
$[0,T)$ with initial datum $\mathcal{N}_0$.
Suppose that the Multiplicity--One Conjecture holds and that $T<+\infty$.
Then, the curves $\{\gamma^i\}_{i=1}^N$ of the flow $\mathcal{N}_t$, 
up to reparametrization proportional to arclength, converge in $C^1\loc$
either to constant maps or to regular $C^2$ curves composing the limit 
network $\widehat{\mathcal{N}}_T$, as $t\to T$.
The only possible singularities, are given by the collapse of isolated ``inner'' curves of the network, producing a regular $4$--point, or the collapse of some ``boundary curves'' on the fixed end--points of the network, letting two concurring curves forming at such end--point an angle of $120$ degrees.
\end{thm}

We mention that then one can apply the results in~\cite{IlNevSchu,LiMaPluSa} and ``restart'' the flow with a smooth evolving regular network, hence the main question that remains open in order to obtain a global existence result, is the possible accumulation in time of the singularities, that we will not address in this paper.

\section*{Acknowledgments}

The authors were partially supported by INDAM--GNAMPA project 2020 
{\it Minimal clusters and minimal partitions},
the first author by the PRIN Project {\it Geometric flows, optimal transport and metric
measure structures} and the second and third authors
by the PRIN Project  {\it Variational methods for stationary and evolution problems with singularities and interfaces}.

\section{Preliminaries}

\subsection{Notation}

Consider a curve of class $C^1([0,1];\mathbb{R}^2)$.
We say that $\gamma$ is {\em regular} if for every $x\in [0,1]$ we have
$\partial_x\gamma(x)\neq 0$. 
It is then well defined
its {\em unit tangent vector} $\tau=\partial_x\gamma/\vert\partial_x\gamma\vert$. 
We define its {\em  unit normal vector} as
$\nu=\mathrm{R}\tau$
where $\mathrm{R}:\mathbb{R}^{2}\to\mathbb{R}^{2}$ is the counterclockwise
rotation centered in the origin of $\mathbb{R}^{2}$ of angle
${\pi}/{2}$.\\

Assuming that  the 
curve is of class $C^2([0,1];\mathbb{R}^2)$.
we define the curvature vector as
\begin{equation*}
\vec{k}(x):=\frac{\partial_x^2\gamma(x)}{\vert\partial_x\gamma(x)\vert^2}
-\frac{\partial_x\gamma(x)\left\langle\partial_x^2\gamma(x),\partial_x\gamma(x)\right\rangle}{\vert\partial_x(x)\gamma\vert^4}
=\kappa(x)\nu(x)\,,
\end{equation*}
where $\kappa$ is the scalar curvature

The {\em arclength parameter} of a curve $\gamma$ is given by
$$
s=s(x)=\int_0^x\vert\partial_x\gamma(\xi)\vert\,\mathrm{d}\xi\,.
$$

We remind here that 
in the whole paper we will use the word ``curve" both for the parametrization and for the
set image of the parametrization in $\mathbb{R}^2$.

\begin{defn}
A network $\mathcal{N}$ is a connected  set in the Euclidean plane, composed of finitely many 
regular, embedded and sufficiently smooth curves that meet only at their end--points.
\end{defn}

We distinguish between interior and exterior vertices of the network:
at the firsts, more then one curve concur,
and the latter are  the termini of the network.

We say that a network is a \emph{tree} if it does not contain loops.

\medskip

We assume that 
each curve of a network admits a regular parametrization $\gamma\in C^2([0,1];\mathbb{R}^2)$.

All along the paper we denote by $L^i$ the length of the $i$--th curve of the network, namely
$$
L^i:= \int_0^1\vert\partial_x\gamma^i(x)\vert\,\mathrm{d}x=\int_{\gamma^i}1\,\mathrm{d}s\,.
$$

We are going to give a name to the following special sub--class of networks
because of the key role they are going to play in the paper:

\begin{defn}\label{regularnet}
A network is regular if it possesses only triple junctions, where  the unit tangent vectors
form angles of $120$ degrees.
\end{defn}

Consider a regular network 
composed of $N$ curves $\gamma^i_\in C^2([0,1],\mathbb{R}^2)$,
that meet at $m$ triple junctions $\mathcal{O}^1, \dots \mathcal{O}^m$ 
and possibly with $\ell$ external vertices $P^1, \dots, P^\ell\in\partial\Omega$.
For every $p\in\{1,\ldots,m\}$
we denote by $\gamma^{p1},\gamma^{p2}$ and $\gamma^{p3}$ the curves that
meet at the junction $\mathcal{O}^p$.
We assumed conventionally (possibly relabeling the family of
curves and inverting their parametrization) that the end--point
$P^r$ of the network is given by $\gamma^r(1,t)$
and we write $\tau^{pj}(O^p)$
for  the outward unit tangent vectors of the three curves
$\gamma^{pj}(\cdot)$ that meet at $\mathcal{O}^p$.

\subsection{Motion by curvature of networks}

\medskip

We are interested in the motion by curvature of networks and to deal with this problem
we use a direct PDE approach.
The network flow can be understood as the gradient flow of the length
\begin{equation*}
L(\mathcal{N}):=\sum_{i} \int_0^1\vert\partial_x\gamma^i(x)\vert\,\mathrm{d}x=\sum_{i} \int_{\gamma^i}1\,\mathrm{d}s\,.
\end{equation*}
Formally, we derive the motion equations computing the first variation of $L$. Each curve moves with normal velocity equal to its curvature
\begin{equation*}
v^\perp(t,x)=\vec{k}(x)\,,
\end{equation*}
or, equivalently $\left\langle\partial_t\gamma(t,x),\nu(t,x)\right\rangle=\kappa(t,x)$.
Moreover, because of the gradient flow structure of the evolution, 
we expect that for almost all the times during the evolution, the network presents only triple junctions
where the unit tangent vectors sum up to zero.
This justify the introduction of the notion of regular networks.

Let $\Omega$ be a convex, open, bounded and regular set in $\mathbb{R}^2$.
We consider the evolution of regular networks in $\Omega$ and we suppose that
the exterior vertices have order one and remain fixed on $\partial\Omega$ during the flow.
Both the case in which the end--points move on $\partial\Omega$ 
(with suitable Neumann boundary conditions) and 
the case of networks with non compact branches
(satisfying suitable conditions at infinity) can be treat similarly, but we do not consider them in this paper.

\begin{defn}
Let $\Omega$ be a bounded, convex, open and regular subset of $\mathbb{R}^2$.
Given an initial regular network $\mathcal{N}_0$, 
composed of $N$ curves $\gamma^i_0\in C^2([0,1],\overline{\Omega})$,
that meet at $m$ triple junctions $\mathcal{O}^1, \dots \mathcal{O}^m$ in $\Omega$ 
and possibly with $\ell$ external vertices $P^1, \dots, P^\ell\in\partial\Omega$,
we say that a family of homeomorphic 
networks $\mathcal{N}_t$ 
is a solution to the motion by curvature in $[0,T)$ with initial datum $\mathcal{N}_0$
if there exist $N$ regular curves parametrized by
$$
\gamma^i\in C^{1,2}([0,T)\times [0,1])\,,
$$
that satisfy
the following system of conditions for every $x\in[0,1]$ and $t\in [0,T)$:
\begin{equation}\label{motionsystem}
\begin{cases}
\begin{array}{ll}
\left\langle\partial_t\gamma^i(t,x),\nu^i(t,x)\right\rangle=\kappa^i(t,x)\quad& i\in\{1,\dots, N\}\\
\gamma^r(1,t)=P^r\quad& r\in\{1,\dots, \ell\}\\
\sum_{j=1}^3\tau^{pj}(O^p,t)=0\quad&\text{at every triple junction $O^p$}
\quad \\
\gamma^i(0,x)=\gamma_0^i(x)
\quad & i\in\{1,\dots, N\},\\
\end{array}
\end{cases}
\end{equation}
where we assumed conventionally (possibly relabeling the family of
curves and inverting their parametrization) that the end--point
$P^r$ of the network is given by $\gamma^r(1,t)$.
\end{defn}

Note that in the third equation  we write $\tau^{pj}(O^p,t)$
for  the outward unit tangent vectors of the three curves
$\gamma^{pj}(\cdot,t)$ that meet at $\mathcal{O}^p$.

\medskip

We underline the fact that the motion equation is equivalent to
\begin{align*}
\partial_t\gamma^i(t,x)=\kappa^i(t,x)\nu^i(t,x)+\zeta^i(t,x)\tau^i(t,x)\,,
\end{align*}
where $\zeta$ is a non--zero tangential velocity, 
at the junctions
geometrically determined by the curvatures of the concurring curves.
In the whole paper we will label the tangential velocity  by $\zeta$.
When $\zeta=\left\langle\frac{\partial_x^2\gamma}{\vert\partial_x\gamma\vert^2},\tau\right\rangle$,
we obtain the so--called special flow 
$$
\partial_t\gamma=\frac{\partial_x^2\gamma}{\vert\partial_x\gamma\vert^2}\,.
$$

\subsection{Short time existence and behavior at the maximal time of existence}

Combining results from~\cite{BroRei,GoMePlu,ManNovPluSchu} one gets:

\begin{thm}\label{short-time}
Let $\mathcal{N}_0$ be a regular network of class $C^2$.
Then there exists a maximal solution 
to the network flow with initial datum $\mathcal{N}_0$
in the maximal time interval $[0,T_{\max})$. It is unique up to reparametrization 
and it admits a smooth parametrization for all positive times.
\end{thm}
\begin{proof}
The main result contained in~\cite{BroRei} is existence and uniqueness 
of the flow in a small time interval $[0,T]$ provided the initial datum
is in the H\"{o}lder space  $C^{2+\alpha}$, with $\alpha\in (0,1)$, with the property that
the curvatures sum up to zero at the junctions.
Then, by an approximation procedure exploit in~\cite[Theorem 6.8]{ManNovPluSchu}, 
it is possible to produce a solution when the initial datum is of class $C^2$,
without any assumption on the curvature at the junctions.
To conclude, uniqueness of solution in a short time interval 
 is proved in~\cite[Theorem 1.1]{GoMePlu}
whenever the initial datum is in the Sobolev--Slobodeckij  space $W^{2-\frac{2}{p},p}$
with $p\in (3,\infty)$. 
So in particular the uniqueness still holds true if $\mathcal{N}_0$ is of class $C^2$.
Existence and uniqueness of a maximal solution in $[0,T_{\max})$ is then easy to get
(see for instance~\cite{GoMePlu}).
Smoothness of solutions for positive times is also part of ~\cite[Theorem 1.1]{GoMePlu}.
\end{proof}

A short time existence result can be shown, with some appropriate weak definitions, also to initial $C^2$ networks that are not regular in the sense of Definition~\ref{regularnet} (see~\cite{IlNevSchu, LiMaPluSa}). However, we underline that whatever definition of curvature flow for a general network (even asking that the networks of the flow are regular for every small positive time), uniqueness can be lost, as it is shown in the following example.
\begin{figure}[H]
\begin{center}
\begin{minipage}[c]{.35\textwidth}
\begin{tikzpicture}
\draw[color=black,scale=1,domain=-3.141: 3.141,
smooth,variable=\t,shift={(-5,2)},rotate=0]plot({1.42*sin(\t r)},
{1.42*cos(\t r)});
\draw[color=black,thick,scale=1,domain=0: 1.5708,
smooth,variable=\t,shift={(-6,2)},rotate=0]plot({1*sin(\t r)},
{1*cos(\t r)});
\draw[color=black,thick,scale=1,domain=0: 1.5708,
smooth,variable=\t,shift={(-4,2)},rotate=180]plot({1*sin(\t r)},
{1*cos(\t r)});
\draw[color=black,thick,scale=1,domain=0: 1.5708,
smooth,variable=\t,shift={(-5,1)},rotate=90]plot({1*sin(\t r)},
{1*cos(\t r)});
\draw[color=black,thick,scale=1,domain=0: 1.5708,
smooth,variable=\t,shift={(-5,3)},rotate=-90]plot({1*sin(\t r)},
{1*cos(\t r)});
\path[font=\large]
(-6,3.1) node[left] {$P^4$}
(-4,3.1) node[right] {$P^3$}
(-6,.9) node[left] {$P^1$}
(-4,.9) node[right] {$P^2$}
(-5.3,2.47) node[below] {$O$};
\draw[color=black!40!white,rotate=69,shift={(.2,.3)}]
(-0.05,2.65)to[out= -90,in=150, looseness=1] (0.17,2.3)
(0.17,2.3)to[out= -30,in=100, looseness=1] (-0.12,2)
(-0.12,2)to[out= -80,in=40, looseness=1] (0.15,1.7)
(0.15,1.7)to[out= -140,in=90, looseness=1.3](0,1.1)
(0,1.1)--(-.2,1.35)
(0,1.1)--(+.2,1.35);
\draw[color=black!40!white,rotate=111,shift={(3.35,-1.05)}]
(-0.05,2.65)to[out= -90,in=150, looseness=1] (0.17,2.3)
(0.17,2.3)to[out= -30,in=100, looseness=1] (-0.12,2)
(-0.12,2)to[out= -80,in=40, looseness=1] (0.15,1.7)
(0.15,1.7)to[out= -140,in=90, looseness=1.3](0,1.1)
(0,1.1)--(-.2,1.35)
(0,1.1)--(+.2,1.35);
\end{tikzpicture}
\end{minipage}
\quad  
\begin{minipage}[c]{.25\textwidth}
\begin{tikzpicture}
\draw[color=black,scale=1,domain=-3.141: 3.141,
smooth,variable=\t,shift={(1,0)},rotate=0]plot({1.42*sin(\t r)},
{1.42*cos(\t r)});
\draw[color=black,thick,scale=1,domain=0: 1.5708,
smooth,variable=\t,shift={(0,0)},rotate=0]plot({1*sin(\t r)},
{1*cos(\t r)});
\draw[color=black,thick,scale=1,domain=0: 1.5708,
smooth,variable=\t,shift={(2,0)},rotate=180]plot({1*sin(\t r)},
{1*cos(\t r)});
\draw[color=black,thick,scale=1,domain=0: 1.5708,
smooth,variable=\t,shift={(1,-1)},rotate=90]plot({1*sin(\t r)},
{1*cos(\t r)});
\draw[color=black,thick,scale=1,domain=0: 1.5708,
smooth,variable=\t,shift={(1,1)},rotate=-90]plot({1*sin(\t r)},
{1*cos(\t r)});
\filldraw[color=white,scale=0.2,domain=-3.141: 3.141,
smooth,variable=\t,shift={(5,0)},rotate=0]plot({1.42*sin(\t r)},
{1.42*cos(\t r)});
\draw[color=black, thick, shift={(1,0)}]
(-0.05,-0.05)to[out=45,in=-135, looseness=1](0.05,0.05);
\draw[color=black, thick, shift={(1,0)}]
(-0.05,-0.05)to[out=165,in=10, looseness=1](-0.29,-0.045);
\draw[color=black, thick, shift={(1,0)}]
(-0.05,-0.05)to[out=-75,in=110, looseness=1](0.045,-0.29);
\draw[color=black, thick, shift={(1,0)},rotate=180]
(-0.05,-0.05)to[out=165,in=10, looseness=1](-0.29,-0.045);
\draw[color=black, thick, shift={(1,0)}, rotate=180]
(-0.05,-0.05)to[out=-75,in=110, looseness=1](0.045,-0.29);
\path[font=\large]
(0,1.1) node[left] {$P^4$}
(2,1.1) node[right] {$P^3$}
(0,-1.1) node[left] {$P^1$}
(2,-1.1) node[right] {$P^2$}
(1.4,0.67) node[below] {$O^1$}
(0.7,-0.13) node[below] {$O^2$};
\end{tikzpicture}

\quad

\begin{tikzpicture}[rotate=90]
\draw[color=black,scale=1,domain=-3.141: 3.141,
smooth,variable=\t,shift={(1,0)},rotate=0]plot({1.42*sin(\t r)},
{1.42*cos(\t r)});
\draw[color=black,thick,scale=1,domain=0: 1.5708,
smooth,variable=\t,shift={(0,0)},rotate=0]plot({1*sin(\t r)},
{1*cos(\t r)});
\draw[color=black,thick,scale=1,domain=0: 1.5708,
smooth,variable=\t,shift={(2,0)},rotate=180]plot({1*sin(\t r)},
{1*cos(\t r)});
\draw[color=black,thick,scale=1,domain=0: 1.5708,
smooth,variable=\t,shift={(1,-1)},rotate=90]plot({1*sin(\t r)},
{1*cos(\t r)});
\draw[color=black,thick,scale=1,domain=0: 1.5708,
smooth,variable=\t,shift={(1,1)},rotate=-90]plot({1*sin(\t r)},
{1*cos(\t r)});
\filldraw[color=white,scale=0.2,domain=-3.141: 3.141,
smooth,variable=\t,shift={(5,0)},rotate=0]plot({1.42*sin(\t r)},
{1.42*cos(\t r)});
\draw[color=black, thick, shift={(1,0)}]
(-0.05,-0.05)to[out=45,in=-135, looseness=1](0.05,0.05);
\draw[color=black, thick, shift={(1,0)}]
(-0.05,-0.05)to[out=165,in=10, looseness=1](-0.29,-0.045);
\draw[color=black, thick, shift={(1,0)}]
(-0.05,-0.05)to[out=-75,in=110, looseness=1](0.045,-0.29);
\draw[color=black, thick, shift={(1,0)},rotate=180]
(-0.05,-0.05)to[out=165,in=10, looseness=1](-0.29,-0.045);
\draw[color=black, thick, shift={(1,0)}, rotate=180]
(-0.05,-0.05)to[out=-75,in=110, looseness=1](0.045,-0.29);
\path[font=\large,rotate=-90, shift={(-1,1)}]
(0,1.1) node[left] {$P^4$}
(2,1.1) node[right] {$P^3$}
(0,-1.1) node[left] {$P^1$}
(2,-1.1) node[right] {$P^2$};
\path[font=\large]
(1.6,0.5) node[below] {$O^1$}
(1,-0.43) node[below] {$O^2$};
\end{tikzpicture}
\end{minipage}
\end{center}
\begin{caption}{An example of non--uniqueness of the flow.\label{nonuniqnet}}
\end{caption}
\end{figure}
\noindent Indeed, by the symmetry of the initial network with respect to rotations of $90$ degrees, the rotation of any possible  evolution must still be a solution, as this invariance property must be clearly satisfied.

\medskip

Once established the existence of a solution in a maximal time interval $[0,T)$ it is natural
to wonder what happens at $T$. One can find a first answer in~\cite{MaManNov,ManNovTor}, that we state in the following proposition.

\begin{prop}\label{longtime}
Let $\mathcal{N}_t$ be the maximal solution to the motion by curvature of networks in the 
time interval $[0,T)$.
Then either
$$
T=+\infty
$$
or at least one of the following properties holds:
\begin{itemize}
\item the inferior limit of the length of at least one curve of the network $\mathcal{N}_t$
is zero as $t\to T$;
\item the superior limit of the integral of the squared curvature is $+\infty$ as $t\to T$.
\end{itemize}
Moreover if the lengths of all the curves of the network are uniformly 
bounded from below during the flow, 
then the superior limit is  a limit, and  for every $t\in[0, T)$
there exists a positive constant $C$ such that 
$$
\int_{{\mathcal{N}_t}} \kappa^2\,\mathrm{d}s \geq \frac{C}{\sqrt{T-t}}\qquad\text{ and }\qquad
\max_{\mathcal{N}_t}\kappa^2\geq\frac{C}{\sqrt{T-t}}\,.
$$
\end{prop}

\subsection{Monotonicity formula and Multiplicity--One--Conjecture}

The problem we tackle now is analyze 
possible singularities arising at the maximal time of existence. To this aim we introduce a rescaling procedure and a monotonicity formula, adapted to our situation and
originally due to Huisken, that enable us to better describe the singularities.

\medskip

To determine and predict how singularities form, homothetically shrinking solutions play an important role.

\begin{defn}
A time--dependent family of networks $\mathcal{N}_t$ parametrized by $\{\gamma^i\}_{i=1}^N$ is a shrinking self--similar
solution to the motion by curvature if each evolving curve has the form
\begin{equation}\label{selfsim}
\gamma^i(t,x)=\lambda(t)\eta^i(x)\,,
\end{equation}
with $\lambda(t)>0$ and $\lambda(t)'<0$.
\end{defn}

Plugging our ansatz~\eqref{selfsim} into the motion equation $\left\langle\partial_t\gamma(t,x),\nu(t,x)\right\rangle=\kappa(t,x)$ and fixing $\lambda(-\frac{1}{2})=1$ it is easy to see that
\begin{equation}
\lambda(t)=\sqrt{-2t}\,,\quad \kappa^i(x)+\left\langle\eta^i(x),\nu^i(x)\right\rangle=0\,.\label{shrinkereq}
\end{equation}

\begin{defn}
The last relation is called the shrinker equation.
A (static) network whose curves satisfy the  equation
$$
\kappa^i(x)+\left\langle\eta^i(x),\nu^i(x)\right\rangle=0
$$
is called shrinker.
\end{defn}

\medskip

Let us introduce a parabolic rescaling procedure and the notion of reachable points
(see~\cite[Section 7, Section 10]{ManNovPluSchu} for more details).

We can describe the solution to the motion by curvature 
as a map from a reference network $\mathcal{N}$ to $\mathbb{R}^2$.
For the time being  let $F:\mathcal{N}\times [0,T)\to\mathbb{R}^2$
be the curvature flow of a regular network in its maximal time interval of existence. 

\medskip

We define the set of {\em reachable points} of the flow as:
$$
\mathcal{R} = \bigl\{ p \in \R^2\,\bigl\vert\,\text{ there exist $p_i
\in \mathcal{N}$ and $t_i \nearrow T$ such that $\lim_{i \to \infty}F(p_i,  t_i) = p$}\bigr\}\,.
$$
Such a set is easily seen to be closed and contained in
$\overline{\Omega}$ (hence compact as $\Omega$ is bounded).
Moreover a point $p \in \R^2$ belongs to $\mathcal{R}$ if and only if for every
time $t \in [0,T)$ the closed ball with center $x$ and radius
$\sqrt{2(T-t)}$ intersects $\mathcal{N}_t$.

\medskip

\begin{defn}\label{pararesc}
For a fixed $\lambda > 0$  we define the 
parabolic rescaling around a
space--time point $(p_0,t_0)$ of $F$ as the family of maps
\begin{equation*}
F^\lambda_\tau = \lambda \big(F(\cdot, \lambda^{-2}\tau + t_0) - p_0\big)\,, 
\end{equation*}
where $\tau \in [-\lambda^2 t_0, \lambda^2(T-t_0))$. 
\end{defn}

After this rescaling we still have that $F^\lambda_\tau $ is a 
curvature flow in the domain $\lambda (\Omega -p_0)$ with a new time
parameter $\tau$. 
With a slight abuse of notation we relabel it as $\mathcal{N}^\lambda_\tau$.

\medskip

Let $(p_0,t_0)\in \mathbb{R}^2\times (0,+\infty)$  and
$\rho_{p_0,t_0}:\mathbb{R}^2\times[-\infty,t_0)\to\mathbb{R}$ be the one--dimensional 
backward   heat kernel in $\mathbb{R}^2$ relative to $(p_0,t_0)$, that is,
\begin{equation*}
\rho_{p_0,t_0}(p,t)=\frac{e^{-\frac{\vert p-p_0\vert^2}{4(t_0-t)}}}{\sqrt{4\pi(t_0-t)}}\,.
\end{equation*}

\begin{defn}
For every $p_0\in\R^2, t_0 \in(0,+\infty)$ we define the Gaussian density function $\Theta_{p_0,t_0}:[0,\min\{t_0,T\})\to\R$ as 
\begin{equation*}
\Theta_{p_0,t_0}(t)=\int_{\mathcal{N}_t}\rho_{p_0,t_0}(\cdot,t)\,\mathrm{d}s\,, 
\end{equation*}
and, provided $t_0\leq T$, the limit Gaussian density function 
$\widehat{\Theta}:\R^2\times (0,+\infty)\to\R$ as
\begin{equation*}
\widehat{\Theta}(p_0,t_0)=\lim_{t\to t_0}\Theta_{p_0,t_0}(t)\,.
\end{equation*}
which exists, is finite and nonnegative for every $(p_0,t_0)\in\R^2\times(0,T]$.
\end{defn}

Given a sequence $\lambda_i\nearrow +\infty$ and a space--time point
$(p_0,t_0)$, where $0<t_0\leq T$, we then consider the sequence of parabolically rescaled
curvature flows $\mathcal{N}^{\lambda_i}_\tau$ in the whole $\R^2$
(see Definition~\ref{pararesc}).
Suitably adapting Huisken's monotonicity formula for this network flows 
and changing variables according to the parabolic rescaling introduced above we get
\begin{align*}
\Theta_{p_0,t_0}(t_0+\lambda_i^{-2}\tau)-\widehat\Theta(p_0,t_0)
=&\,\int_{\tau}^{0}\int
_{\mathcal{N}^{\lambda_i}_\sigma}\;
\Big|\vec{k}^i-\frac{x^\perp}{2\sigma}\Big|^2\rho_{0,0}(\cdot,\sigma)\,
\mathrm{d}s\,\mathrm{d}\sigma\\ 
&\,+\sum_{r=1}^\ell\int_{\tau}^{0} 
\left[ \left\langle\frac{P_i^r}{2\sigma}\,\biggr\vert\,
\tau(P_i^r,\sigma)\right\rangle
\right]
\rho_{0,0}(P_i^r,\sigma)\,\mathrm{d}\sigma\,,
\end{align*}
where $P^r_i = \lambda_i(P^r-p_0)$ and $\vec{k}^i$ are the rescaled curvatures.
Sending $i\to\infty$ for every $\tau\in(-\infty, 0)$ one gets
\begin{equation*}
\lim_{i\to\infty}\,\int_{\tau}^{0}\int
\limits_{\mathcal{N}^{\lambda_i}_\sigma}
\Big|\vec{k}^i- \frac{x^\perp}{2\sigma}\Big|^2\rho_{0,0}(\cdot,\sigma)\,
\mathrm{d}s\,\mathrm{d}\sigma=0\,.
\end{equation*}
This result is non trivial, also because of the presence of the boundary terms,
we refer to~\cite[Section 7]{ManNovPluSchu} for further details.

The monotonicity formula is the key ingredient to obtain the following:

\begin{prop}\label{convergenza-a-shrinkers}
Let $\mathcal{N}^{\lambda_i}_\tau$ be a sequence of parabolically rescaled curvature flows (see Definition~\ref{pararesc}).
Then  $\mathcal{N}^{\lambda_i}_\tau$ converges (up to subsequences)
 to a (possibly empty) 
self--similarly shrinking network flow $\mathcal{N}^\infty_\tau$ 
and   for  almost all $\tau \in (-\infty, 0)$ 
in $C^{1,\alpha}\loc \cap W^{2,2}\loc$  for any $\alpha \in (0,1/2)$ and 
 in the sense of measures for all $\tau \in (-\infty, 0)$ .
\end{prop}
\begin{proof}
See~	\cite[Proposition 8.16]{ManNovPluSchu}.
\end{proof}

The flows obtained as a limit of parabolically rescaled curvature flows
(see Definition~\ref{pararesc}).
go under the name of tangent flow.
Now we would like to classify the possible
self--similarly shrinking network flows
arising as limit of parabolically rescaled curvature flows.
We state now a conjecture about their nature.

\medskip

\textbf{Multiplicity--One Conjecture (M1)}:
Every limit of parabolic rescalings at a point $p_0\in\overline{\Omega}$ 
is a flow of embedded networks with multiplicity one.

\begin{thm}
Let $\mathcal{N}_0$ be a initial regular network of class $C^2$ in 
a strictly convex, bounded and open set
$\Omega\subseteq\mathbb{R}^2$  and
let $\mathcal{N}_t$ be a maximal solution to the motion by curvature in $[0,T)$
with initial datum $\mathcal{N}_0$. 
Suppose that
\begin{itemize}
\item either $\mathcal{N}_0$ has at most two triple junctions;
\item or $\mathcal{N}_0$ is a tree and 
if for all $t\in [0,T)$
the triple junctions  of $\mathcal{N}_t$ 
remains uniformly far from each others and from the vertices fixed on 
$\partial\Omega$.
\end{itemize}
Then, the Multiplicity--One Conjecture holds true.
\end{thm}
\begin{proof}
See~\cite[Corollary 4.7]{ManNovPlu} and~\cite[Proposition 14.13]{ManNovPluSchu}
\end{proof}

\section{Analysis of tangent flows}\label{blowuplim}

From now on we shall always assume that 
the Multiplicity--One Conjecture holds true.

\begin{defn}
A standard triod is a network composed of 
three straight halflines that meet at the origin forming angles of $120$ degrees.
A standard cross is the union of two lines intersecting at the origin forming angles of $120$
and $60$ degrees.
\end{defn}

\begin{figure}[H]
\begin{center}
\begin{tikzpicture}[scale=0.7]
\draw
(0,0) to [out=90, in=-90, looseness=1] (0,2)
(0,0) to [out=210, in=30, looseness=1] (-1.73,-1)
(0,0) to [out=-30, in=150, looseness=1](1.73,-1);
\draw[dashed]
(0,2) to [out=90, in=-90, looseness=1] (0,3)
(-1.73,-1) to [out=210, in=30, looseness=1] (-2.59,-1.5)
(1.73,-1) to [out=-30, in=150, looseness=1](2.59,-1.5);
\fill(0,0) circle (2pt);
\path[font=\normalsize]
(.1,.2) node[left]{$\mathcal{O}$};
\end{tikzpicture}\qquad\qquad\qquad
\begin{tikzpicture}[scale=0.3]
\draw[color=black]
(0,0)to[out= 120,in=-60, looseness=1] (-2,3.46)
(0,0)to[out= -120,in=60, looseness=1] (-2,-3.46)
(0,0)to[out= 60,in=-120, looseness=1] (2,3.46)
(0,0)to[out=-60,in=120, looseness=1] (2,-3.46);
\draw[color=black,dashed]
(-2,3.46)to[out= 120,in=-60, looseness=1] (-3,5.19)
(-2,-3.46)to[out= -120,in=60, looseness=1] (-3,-5.19)
(2,3.46)to[out= 60,in=-120, looseness=1] (3,5.19)
(2,-3.46)to[out=-60,in=120, looseness=1] (3,-5.19);
\fill(0,0) circle (4pt);
\path[font=\normalsize]
(1.1,-.35) node[above]{$\mathcal{O}$};
\end{tikzpicture}
\end{center}
\begin{caption}{A standard triod and a standard cross.}\label{standard}
\end{caption}
\end{figure}

\begin{lem}\label{lemmatree}
Let $\mathcal{S}$ be a shrinker which is $C^1\loc\cap W^{2,2}\loc$--limit of
a sequence $\mathcal{N}_n$
of homeomorphic regular tree--like connected networks
with non--compact branches and without external vertices. 
Then $\mathcal{S}$ consists of halflines from the origin.\\
Moreover, if we assume that $\mathcal{S}$ is a network with unit multiplicity, 
then $\mathcal{S}$ can only be
\begin{itemize}
\item a line,
\item a standard triod,
\item  a standard cross.
\end{itemize}
\end{lem}
\begin{proof} 
First of all we notice that the length of some curves can go to zero in the limit.
The \emph{core} of the limit network is the union of these vanishing curves. 

Suppose that the length of
at least one curve of $\mathcal{N}_n$, let us say $\gamma_n$, goes to zero  in the limit,
namely $\mathcal{S}$ has a core at some point $P\in\mathcal{S}$.
Since $\mathcal{N}_n$ is a sequence of trees,
if $N\geq 2$ triple junctions are contained in the core, then $N+2$
curves (counted with multiplicity) with strictly positive length concur at $P$.
This fact can be easily proved by induction: if $N=2$, then two triple junctions
are present in the core
and hence the length
of the curve connecting the two junctions has gone to zero in the limit,
but the other four curves emanating from the two different junctions have still positive length. Suppose now that the statement holds true for $N=\widetilde{N}$. 
Let now $N=\widetilde{N}+1$. With respect to the situation in which
$\widetilde{N}$ triple junctions are in the core, we 
add an extra triple junction $\mathcal{O}$ to the core, but
to do so one of the original $\widetilde{N}+2$ curves emanating from the core
has to go to zero. However 
the other two concurring curve to $\mathcal{O}$ have length bounded from below away from zero and now concur to $P$, thus there are 
$(\widetilde{N}+2)-1+2=\widetilde{N}+3$ curves with strictly positive length concurring at $P$
and the claim is proved.

We can suppose (up to reparametrization)
that for every $n\in\mathbb{N}$ the immersion $\gamma_n:[0,1]\to\mathbb{R}^2$
is a regular parametrization with constant speed equal to its length.
We get
\begin{equation*}
\lim_{n\to\infty}\,\sup_{x,y\in [0,1]} \left\lvert 
{\tau_n(x)-\tau_n(y)}\right\rvert=0\,.
\end{equation*}
Indeed,  given $x,y\in[0,1]$, we get
$$
 \left\lvert 
{\tau_n(x)-\tau_n(y)}\right\rvert= \left| \int_{s(x)}^{s(y)} \partial_s\tau_{n}\,\mathrm{d}s\,\right|\leq  \int_{\gamma_n} |\vec{k}_n| \, \mathrm{d}s\leq \left( \int_{\gamma_n} |\vec{k}_n|^2 \, \mathrm{d}s \right)^{1/2} L(\gamma_n)^{1/2}\,.
$$

We then obtain the desired result by passing to the limit.
Hence the vanishing curves of $\mathcal{N}_n$ are straighter and straighter
as $n\to\infty$ and for $n$ big enough we can suppose that the unit tangent vectors
of points of the same curve
does not change direction.

We describe now the structure of the core.
Let $n\in\mathbb{N}$ be sufficiently big and 
consider the the longest simple ``path'' of curves of  $\mathcal{N}_n$ that belongs to the core 
of $\mathcal{S}$ at $P$. We orient the path and follow its edges.
Then unit tangent vectors (possibly changed of sign on some edges in order to coincide with the orientation of the path) cannot turn of an angle of $60$ degrees in the same direction for two consecutive times, otherwise, since $\mathcal{N}_n$ is a sequence of trees with triple junctions, without external vertices and with non--compact branches, $\mathcal{N}_n$ must have a self--intersection (see Figure~\ref{fig6}). 
\begin{figure}[H]
\begin{center}
\begin{tikzpicture}[scale=0.3,rotate=90]
\draw[color=black]
(-1,0)to[out= 0,in=180, looseness=1](1,0)
(3,-3.46)to[out= 0,in=180, looseness=1](5,-3.46)
(-3,-3.46)to[out= 0,in=180, looseness=1](-5,-3.46)
(-1,0)to[out= 120,in=-60, looseness=1] (-3,3.46)
(-1,0)to[out= -120,in=60, looseness=1] (-3,-3.46)
(1,0)to[out= 60,in=-120, looseness=1] (3,3.46)
(1,0)to[out=-60,in=120, looseness=1] (3,-3.46)
(-3,-3.46)to[out=-60,in=120, looseness=1] (-0.12,-8.45)
(0.1,-8.823)to[out=-60,in=120, looseness=1] (1,-10.38)
(3,-3.46)to[out=-120,in=60, looseness=1] (-1,-10.38);
\draw[color=black,dashed]
(5,-3.46)to[out= 0,in=180, looseness=1](6.4,-3.46)
(-5,-3.46)to[out= 0,in=180, looseness=1](-6.2,-3.46)
(-3,3.46)to[out= 120,in=-60, looseness=1] (-4,5.19)
(3,3.46)to[out= 60,in=-120, looseness=1] (4,5.19);
\draw[very thick,shift={(1,0)},scale=2,color=black,rotate=210]
(0,0)to[out= 90,in=-90, looseness=1](0,0.75)
(0,0.75)to[out= -45,in=135, looseness=1](0.15,0.6)
(0,0.75)to[out= -135,in=45, looseness=1](-0.15,0.6);
\draw[very thick,shift={(-1,0)},scale=2,color=black,rotate=-90]
(0,0)to[out= 90,in=-90, looseness=1](0,0.75)
(0,0.75)to[out= -45,in=135, looseness=1](0.15,0.6)
(0,0.75)to[out= -135,in=45, looseness=1](-0.15,0.6);
\draw[very thick, shift={(-3,-3.46)},scale=2,color=black,rotate=330]
(0,0)to[out= 90,in=-90, looseness=1](0,0.75)
(0,0.75)to[out= -45,in=135, looseness=1](0.15,0.6)
(0,0.75)to[out= -135,in=45, looseness=1](-0.15,0.6);
\fill(1,0) circle (4pt);
\fill(-1,0) circle (4pt);
\fill(3,-3.46)circle (4pt);
\fill(-3,-3.46)circle (4pt);
\path[font=]
(-5.4,-8) node[below] {$\mathcal{N}_n$};
\end{tikzpicture}\qquad\qquad
\begin{tikzpicture}[scale=0.3,rotate=90]
\draw[color=black]
(0,0)to[out= 0,in=180, looseness=1](3,0)
(0,0)to[out= 0,in=180, looseness=1](-3,0)
(0,0)to[out= 120,in=-60, looseness=1] (-2,3.46)
(0,0)to[out= 60,in=-120, looseness=1] (2,3.46)
(0,0)to[out=-60,in=120, looseness=1](2,-3.46)
(0,0)to[out=-120,in=60, looseness=1] (-2,-3.46);
\draw[color=black,dashed]
(3,0)to[out= 0,in=180, looseness=1](4.5,0)
(-3,0)to[out= 0,in=180, looseness=1](-4.5,0)
(-2,3.46)to[out= 120,in=-60, looseness=1] (-3,5.19)
(2,3.46)to[out= 60,in=-120, looseness=1] (3,5.19)
(-2,-3.46)to[out= -120,in=60, looseness=1] (-3,-5.19)
(2,-3.46)to[out= -60,in=120, looseness=1] (3,-5.19);
\fill(0,0) circle (4pt);
\path[font=]
(-5.4,-8) node[below] {$\mathcal{S}$};
\end{tikzpicture}\qquad
\begin{tikzpicture}[scale=0.3,rotate=90]
\draw[color=black]
(-1,0)to[out= 0,in=180, looseness=1](1,0)
(-1,0)to[out= -120,in=60, looseness=1] (-3,-3.46)
(1,0)to[out=-60,in=120, looseness=1] (3,-3.46);
(-3,-3.46)to[out=-60,in=120, looseness=1] (-0.12,-8.45)
(0.1,-8.823)to[out=-60,in=120, looseness=1] (1,-10.38)
(3,-3.46)to[out=-120,in=60, looseness=1] (-1,-10.38);
\draw[very thick,shift={(1,0)},scale=2,color=black,rotate=210]
(0,0)to[out= 90,in=-90, looseness=1](0,0.75)
(0,0.75)to[out= -45,in=135, looseness=1](0.15,0.6)
(0,0.75)to[out= -135,in=45, looseness=1](-0.15,0.6);
\draw[very thick,shift={(-1,0)},scale=2,color=black,rotate=-90]
(0,0)to[out= 90,in=-90, looseness=1](0,0.75)
(0,0.75)to[out= -45,in=135, looseness=1](0.15,0.6)
(0,0.75)to[out= -135,in=45, looseness=1](-0.15,0.6);
\draw[very thick, shift={(-3,-3.46)},scale=2,color=black,rotate=330]
(0,0)to[out= 90,in=-90, looseness=1](0,0.75)
(0,0.75)to[out= -45,in=135, looseness=1](0.15,0.6)
(0,0.75)to[out= -135,in=45, looseness=1](-0.15,0.6);
\path[font=\footnotesize]
(-4,-5) node[below] {The core of $\mathcal{S}$};
\path[white, font=]
(-5.4,-8) node[below] {$\mathcal{S}$};
\end{tikzpicture}\qquad
\end{center}
\begin{caption}{If the unit tangent vector turns of an angle of $60$ degrees in the same direction for two consecutive times, $\mathcal{N}_n$ has self--intersections.\label{fig6}}
\end{caption}
\end{figure}
Hence when we run across such longest path and we pass from one edge to another,
if the unit tangent vector turns of an angle of $60$ degrees then it must turn back.
This means that at the initial/final point of such path, either the two unit tangent vectors are the same (when the number of edges is odd) or they differ of $60$ degrees (when the number of edges is even). By a simple check, we can then see that, in the first case the four curves
with positive length exiting from such initial/final points of the path, have four different exterior unit tangent vectors at $P$ (opposite in pairs), in the second case, they have three exterior unit tangent vectors at $P$ which are non--proportional each other.
\begin{figure}[H]
\begin{center}
\begin{tikzpicture}[scale=0.30]
\draw[color=black]
(2,-1.73)to[out= 0,in=180, looseness=1](4,-1.73)
(2,1.73)to[out= 0,in=180, looseness=1](4,1.73)
(2,1.73)to[out= 120,in=-60, looseness=1] (1,3.46)
(2,-1.73)to[out= -120,in=60, looseness=1] (1,-3.46)
(1,0)to[out= 60,in=-120, looseness=1] (2,1.73)
(1,0)to[out=-60,in=120, looseness=1] (2,-1.73);
\draw[color=black,dashed]
(4,-1.73)to[out= 0,in=180, looseness=1](5,-1.73)
(4,1.73)to[out= 0,in=180, looseness=1](5,1.73)
(1,3.46)to[out= 120,in=-60, looseness=1] (0,5.19);
(1,-3.46)to[out= -120,in=60, looseness=1] (0,-5.19);
\draw
(1,0)--(-1,0)
(1,-3.46)--(-1,-3.46)
(1,-3.46)--(2,-5.19);
\draw[dashed]
(2,-5.19)--(3,-6.92)
(-1,0)--(-2,0)
(-1,-3.46)--(-2,-3.46);
\fill(1,0) circle (3pt);
\fill(2,1.73) circle (3pt);
\fill(2,-1.73) circle (3pt);
\fill(1,-3.46) circle (3pt);
\path[font=](4,-5.73) node[below] {$\mathcal{N}_n$};
\draw[very thick, shift={(1,0)},scale=1,color=black,rotate=330]
(0,0)to[out= 90,in=-90, looseness=1](0,0.75)
(0,0.75)to[out= -45,in=135, looseness=1](0.15,0.6)
(0,0.75)to[out= -135,in=45, looseness=1](-0.15,0.6);
\draw[very thick, shift={(1,-3.46)},scale=1,color=black,rotate=330]
(0,0)to[out= 90,in=-90, looseness=1](0,0.75)
(0,0.75)to[out= -45,in=135, looseness=1](0.15,0.6)
(0,0.75)to[out= -135,in=45, looseness=1](-0.15,0.6);
\draw[very thick, shift={(2,-1.73)},scale=1,color=black,rotate=-330]
(0,0)to[out= 90,in=-90, looseness=1](0,0.75)
(0,0.75)to[out= -45,in=135, looseness=1](0.15,0.6)
(0,0.75)to[out= -135,in=45, looseness=1](-0.15,0.6);
\end{tikzpicture}\qquad\quad
\begin{tikzpicture}[scale=0.30]
\draw[color=black]
(0,1.73)to[out= 0,in=180, looseness=1](2,1.73)
(2,1.73)to[out= 0,in=180, looseness=1](4,1.73)
(2,1.73)to[out= 120,in=-60, looseness=1] (1,3.46)
(2,1.73)to[out= -60,in=120, looseness=1] (3,0);
\draw[color=black,dashed]
(-1,1.73)to[out= 0,in=180, looseness=1](0,1.73)
(4,1.73)to[out= 0,in=180, looseness=1](5,1.73)
(1,3.46)to[out= 120,in=-60, looseness=1] (0,5.19)
(3,0)to[out= -60,in=120, looseness=1] (4,-1.73);
\fill(2,1.73) circle (3pt);
\path[font=\footnotesize](3,1.53) node[above] {$2$};
\path[font=\footnotesize](1,1.90) node[below] {$2$};
\path[font=\footnotesize](0.6,3.66) node[below] {$1$};
\path[font=\footnotesize](3.4,-.2) node[above] {$1$};
\path[font=](3,-2.27) node[below] {$\mathcal{S}$};
\end{tikzpicture}\qquad\quad
\begin{tikzpicture}[scale=0.30]
\draw[color=black]
(-1,0)to[out= 0,in=180, looseness=1](1,0)
(2,1.73)to[out= 0,in=180, looseness=1](4,1.73)
(2,-1.73)to[out= 0,in=180, looseness=1](4,-1.73)
(2,1.73)to[out= 120,in=-60, looseness=1] (1,3.46)
(2,-1.73)to[out= -120,in=60, looseness=1] (1,-3.46)
(1,0)to[out= 60,in=-120, looseness=1] (2,1.73)
(1,0)to[out=-60,in=120, looseness=1] (2,-1.73);
\draw[color=black,dashed]
(-2,0)to[out= 0,in=180, looseness=1](-1,0)
(4,1.73)to[out= 0,in=180, looseness=1](5,1.73)
(4,-1.73)to[out= 0,in=180, looseness=1](5,-1.73)
(1,3.46)to[out= 120,in=-60, looseness=1] (0,5.19)
(1,-3.46)to[out= -120,in=60, looseness=1] (0,-5.19);
\fill(1,0) circle (3pt);
\fill(2,1.73) circle (3pt);
\fill(2,-1.73) circle (3pt);
\path[font=](3,-4) node[below] {$\mathcal{N}_n$};
\draw[very thick, shift={(1,0)},scale=1,color=black,rotate=330]
(0,0)to[out= 90,in=-90, looseness=1](0,0.75)
(0,0.75)to[out= -45,in=135, looseness=1](0.15,0.6)
(0,0.75)to[out= -135,in=45, looseness=1](-0.15,0.6);
\draw[very thick, shift={(2,-1.73)},scale=1,color=black,rotate=-330]
(0,0)to[out= 90,in=-90, looseness=1](0,0.75)
(0,0.75)to[out= -45,in=135, looseness=1](0.15,0.6)
(0,0.75)to[out= -135,in=45, looseness=1](-0.15,0.6);
\end{tikzpicture}
\qquad\quad
\begin{tikzpicture}[scale=0.30]
\draw[color=black]
(0,1.73)to[out= 0,in=180, looseness=1](2,1.73)
(2,1.73)to[out= 0,in=180, looseness=1](4,1.73)
(2,1.73)to[out= 120,in=-60, looseness=1] (1,3.46)
(2,1.73)to[out= -120,in=60, looseness=1] (1,0);
\draw[color=black,dashed]
(-1,1.73)to[out= 0,in=180, looseness=1](0,1.73)
(4,1.73)to[out= 0,in=180, looseness=1](5,1.73)
(1,3.46)to[out= 120,in=-60, looseness=1] (0,5.19)
(1,0)to[out= -120,in=60, looseness=1] (0,-1.73);
\fill(2,1.73) circle (3pt);
\path[font=\footnotesize](-0.2,1.53) node[above] {$1$};
\path[font=\footnotesize](3.5,1.53) node[above] {$2$};
\path[font=\footnotesize](2,3.46) node[below] {$1$};
\path[font=\footnotesize](2,0) node[above] {$1$};
\path[font=](3,-2.27) node[below] {$\mathcal{S}$};
\end{tikzpicture}\qquad\quad

\end{center}
\begin{caption}{Examples of the edges at the initial/final points of the longest simple path in $\mathcal{N}_n$ and of the relative curves in $\mathcal{S}$,
the numbers $1$ and $2$ denote their multiplicity.}
\end{caption}
\end{figure}

Consider the curves with positive length.
Each curve of the limit satisfies the shrinkers equation~\eqref{shrinkereq}.
They have been classified by Abresch and Langer~\cite{AbLa}.
There are basically two possibilities.
A branch of the network is non compact and it is 
either a straight line through the origin  or a straight halfline
with end--point the origin.
Otherwise, it is contained in a
compact subset of $\R^2$ and has a constant winding direction with
respect to the origin.

Our scope at this point is to exclude the presence of compact branches in $\mathcal{S}$. 
Let $P\not=(0,0)$ and suppose that either $P$ is a triple junction or there 
is a core at some point $P\not=0$.
Then there are at least two distinct non--straight Abresch--Langer curves arriving/starting at $P$.  Indeed, if at $P$ there is a triple junction none of the three curves meeting at $P$
 can be a line trough both the origin and $P$, neither a halfline from the origin and 
trough $P$.
Similarly, if an odd number of curves belongs to the core at $P$
 at most two curves can be such a straight lines/halflines 
 and at most one if an even number of curves belongs to the core.
 Let $\sigma^i$ be one of the non--straight Abresch--Langer curves. We follow $\sigma^i$ 
from $P$ till its other end--point $Q$. At this vertex, even if there is a core at $Q$, there is
always another different non--straight curve $\sigma^{j}$. 
We have that $\sigma^i$ and $\sigma^{j}$ are limit of sequence of curves
$\gamma^i_n$ and $\gamma^j_n$, which are either concurring in $\mathcal{N}_n$
or joint by a path of curves $\gamma^k_n,\ldots, \gamma^l_n$ collapsed to $Q$.	
We continue running along this path in $\mathcal{S}$ till
we arrive at an already considered vertex, which happens since the number of vertices of 
$\mathcal{N}_n$ is finite, obtaining a closed loop, hence, a contradiction. Thus, $\mathcal{S}$ cannot contain $3$--points or cores outside not located at the origin.
Suppose now that $\mathcal{S}$ contains a  non--straight Abresch--Langer curve
with an end-point the origin. 
Since $\mathcal{S}$ is a tree  the other end--point is non at the origin, so 
we can repeat this argument getting again a contradiction, hence, we are done with the first part of the lemma, since then $\mathcal{S}$ can only consist of halflines from the origin.

\medskip

Now we assume that $\mathcal{S}$ is a network with multiplicity one, composed of halflines from the origin.
If there is no core, $\mathcal{S}$ is homeomorphic to 
the whole sequence $\mathcal{N}_n$
 and composed only by halflines from the origin, hence $\mathcal{N}_n$ has at most one vertex, by connectedness. If $\mathcal{N}_n$ has no 
vertices, then $\mathcal{S}$ must be a line, if it has a triple junction, $\mathcal{S}$ is a standard triod.\\
If there is a core at the origin, 
the halflines of $\mathcal{S}$ can only have
six possible directions, by the $120$ degrees condition. 
Hence the sequence $\mathcal{N}_n$ is composed of trees
with at most six unbounded edges. Arguing as in the first part of the
lemma, if $N> 1$ denotes the number of triple junctions
contained in the core, it follows that $N$ can only assume the values
$2, 3, 4$. Repeating the argument of the ``longest path'', we
immediately also exclude the case $N=3$, since there would be a pair
of coincident halflines in $\mathcal{S}$, against the multiplicity--one
hypothesis, while for $N=4$ we have only two possible situations,
described at the bottom of the following figure.
\begin{figure}[H]
\begin{center}
\begin{tikzpicture}[scale=0.30]
\draw[color=black]
(-1,0)to[out= 0,in=180, looseness=1](1,0)
(-1,0)to[out= 120,in=-60, looseness=1] (-3,3.46)
(-1,0)to[out= -120,in=60, looseness=1] (-3,-3.46)
(1,0)to[out= 60,in=-120, looseness=1] (3,3.46)
(1,0)to[out=-60,in=120, looseness=1] (3,-3.46);
\draw[color=black,dashed]
(-3,3.46)to[out= 120,in=-60, looseness=1] (-4,5.19)
(-3,-3.46)to[out= -120,in=60, looseness=1] (-4,-5.19)
(3,3.46)to[out= 60,in=-120, looseness=1] (4,5.19)
(3,-3.46)to[out=-60,in=120, looseness=1] (4,-5.19);
\fill(1,0) circle (4pt);
\fill(-1,0) circle (4pt);
\draw[very thick, shift={(1,0)},color=black,scale=1,rotate=90]
(0,0)to[out= 90,in=-90, looseness=1](0,0.75)
(0,0.75)to[out= -45,in=135, looseness=1](0.15,0.6)
(0,0.75)to[out= -135,in=45, looseness=1](-0.15,0.6);
\path[font=]
(-0,-4) node[below] {$\mathcal{N}_n$};
\end{tikzpicture}\quad
\begin{tikzpicture}[scale=0.30]
\draw[color=black]
(0,0)to[out= 120,in=-60, looseness=1] (-2,3.46)
(0,0)to[out= -120,in=60, looseness=1] (-2,-3.46)
(0,0)to[out= 60,in=-120, looseness=1] (2,3.46)
(0,0)to[out=-60,in=120, looseness=1] (2,-3.46);
\draw[color=black,dashed]
(-2,3.46)to[out= 120,in=-60, looseness=1] (-3,5.19)
(-2,-3.46)to[out= -120,in=60, looseness=1] (-3,-5.19)
(2,3.46)to[out= 60,in=-120, looseness=1] (3,5.19)
(2,-3.46)to[out=-60,in=120, looseness=1] (3,-5.19);
\fill(0,0) circle (4pt);
\path[font=]
(-0,-4) node[below] {$\mathcal{S}$};
\end{tikzpicture}\quad
\begin{tikzpicture}[scale=0.30]
\draw[color=black]
(-1,0)to[out= 0,in=180, looseness=1](1,0);
\draw[very thick, shift={(1,0)},color=black,scale=1,rotate=90]
(0,0)to[out= 90,in=-90, looseness=1](0,0.75)
(0,0.75)to[out= -45,in=135, looseness=1](0.15,0.6)
(0,0.75)to[out= -135,in=45, looseness=1](-0.15,0.6);
\path[font=\footnotesize]
(-0,-4) node[below] {The core of $\mathcal{S}$};
\end{tikzpicture}\qquad
\begin{tikzpicture}[scale=0.30]
\draw[color=black]
(-1,0)to[out= 0,in=180, looseness=1](1,0)
(2,1.73)to[out= 0,in=180, looseness=1](4,1.73)
(2,-1.73)to[out= 0,in=180, looseness=1](4,-1.73)
(2,1.73)to[out= 120,in=-60, looseness=1] (1,3.46)
(2,-1.73)to[out= -120,in=60, looseness=1] (1,-3.46)
(1,0)to[out= 60,in=-120, looseness=1] (2,1.73)
(1,0)to[out=-60,in=120, looseness=1] (2,-1.73);
\draw[color=black,dashed]
(-2,0)to[out= 0,in=180, looseness=1](-1,0)
(4,1.73)to[out= 0,in=180, looseness=1](5,1.73)
(4,-1.73)to[out= 0,in=180, looseness=1](5,-1.73)
(1,3.46)to[out= 120,in=-60, looseness=1] (0,5.19)
(1,-3.46)to[out= -120,in=60, looseness=1] (0,-5.19);
\fill(1,0) circle (4pt);
\fill(2,1.73) circle (4pt);
\fill(2,-1.73) circle (4pt);
\path[font=](3,-4) node[below] {$\mathcal{N}_n$};
\draw[very thick, shift={(1,0)},scale=1,color=black,rotate=330]
(0,0)to[out= 90,in=-90, looseness=1](0,0.75)
(0,0.75)to[out= -45,in=135, looseness=1](0.15,0.6)
(0,0.75)to[out= -135,in=45, looseness=1](-0.15,0.6);
\draw[very thick, shift={(2,-1.73)},scale=1,color=black,rotate=-330]
(0,0)to[out= 90,in=-90, looseness=1](0,0.75)
(0,0.75)to[out= -45,in=135, looseness=1](0.15,0.6)
(0,0.75)to[out= -135,in=45, looseness=1](-0.15,0.6);
\end{tikzpicture}\quad
\begin{tikzpicture}[scale=0.30]
\draw[color=black]
(0,1.73)to[out= 0,in=180, looseness=1](2,1.73)
(2,1.73)to[out= 0,in=180, looseness=1](4,1.73)
(2,1.73)to[out= 120,in=-60, looseness=1] (1,3.46)
(2,1.73)to[out= -120,in=60, looseness=1] (1,0);
\draw[color=black,dashed]
(-1,1.73)to[out= 0,in=180, looseness=1](0,1.73)
(4,1.73)to[out= 0,in=180, looseness=1](5,1.73)
(1,3.46)to[out= 120,in=-60, looseness=1] (0,5.19)
(1,0)to[out= -120,in=60, looseness=1] (0,-1.73);
\fill(2,1.73) circle (4pt);
\path[font=](3,-2.27) node[below] {$\mathcal{S}$};
\end{tikzpicture}\quad
\begin{tikzpicture}[scale=0.30]
\draw[color=black]
(1,0)to[out= 60,in=-120, looseness=1] (2,1.73)
(1,0)to[out=-60,in=120, looseness=1] (2,-1.73);
\draw[very thick, shift={(1,0)},scale=1,color=black,rotate=330]
(0,0)to[out= 90,in=-90, looseness=1](0,0.75)
(0,0.75)to[out= -45,in=135, looseness=1](0.15,0.6)
(0,0.75)to[out= -135,in=45, looseness=1](-0.15,0.6);
\draw[very thick, shift={(2,-1.73)},scale=1,color=black,rotate=-330]
(0,0)to[out= 90,in=-90, looseness=1](0,0.75)
(0,0.75)to[out= -45,in=135, looseness=1](0.15,0.6)
(0,0.75)to[out= -135,in=45, looseness=1](-0.15,0.6);
\path[font=\footnotesize]
(2,-4) node[below] {The core of $\mathcal{S}$};
\end{tikzpicture}\qquad\quad
\begin{tikzpicture}[scale=0.30]
\draw[color=black]
(-1,0)to[out= 0,in=180, looseness=1](1,0)
(2,1.73)to[out= 0,in=180, looseness=1](4,1.73)
(2,-1.73)to[out= 0,in=180, looseness=1](4,-1.73)
(2,1.73)to[out= 120,in=-60, looseness=1] (1,3.46)
(2,-1.73)to[out= -120,in=60, looseness=1] (1,-3.46)
(1,0)to[out= 60,in=-120, looseness=1] (2,1.73)
(1,0)to[out=-60,in=120, looseness=1] (2,-1.73);
\draw[color=black,dashed]
(-2,0)to[out= 0,in=180, looseness=1](-1,0)
(4,1.73)to[out= 0,in=180, looseness=1](5,1.73)
(4,-1.73)to[out= 0,in=180, looseness=1](5,-1.73)
(1,3.46)to[out= 120,in=-60, looseness=1] (0,5.19);
(1,-3.46)to[out= -120,in=60, looseness=1] (0,-5.19);
\draw
(1,-3.46)--(-1,-3.46)
(1,-3.46)--(2,-5.19);
\draw[dashed]
(2,-5.19)--(3,-6.92)
(-1,-3.46)--(-2,-3.46);
\fill(1,0) circle (3pt);
\fill(2,1.73) circle (3pt);
\fill(2,-1.73) circle (3pt);
\fill(1,-3.46) circle (3pt);
\path[font=](4,-5.73) node[below] {$\mathcal{N}_n$};
\draw[very thick, shift={(1,0)},scale=1,color=black,rotate=330]
(0,0)to[out= 90,in=-90, looseness=1](0,0.75)
(0,0.75)to[out= -45,in=135, looseness=1](0.15,0.6)
(0,0.75)to[out= -135,in=45, looseness=1](-0.15,0.6);
\draw[very thick, shift={(1,-3.46)},scale=1,color=black,rotate=330]
(0,0)to[out= 90,in=-90, looseness=1](0,0.75)
(0,0.75)to[out= -45,in=135, looseness=1](0.15,0.6)
(0,0.75)to[out= -135,in=45, looseness=1](-0.15,0.6);
\draw[very thick, shift={(2,-1.73)},scale=1,color=black,rotate=-330]
(0,0)to[out= 90,in=-90, looseness=1](0,0.75)
(0,0.75)to[out= -45,in=135, looseness=1](0.15,0.6)
(0,0.75)to[out= -135,in=45, looseness=1](-0.15,0.6);
\end{tikzpicture}\quad
\begin{tikzpicture}[scale=0.30]
\draw[color=black]
(0,1.73)to[out= 0,in=180, looseness=1](2,1.73)
(2,1.73)to[out= 0,in=180, looseness=1](4,1.73)
(2,1.73)to[out= 120,in=-60, looseness=1] (1,3.46)
(2,1.73)to[out= -60,in=120, looseness=1] (3,0);
\draw[color=black,dashed]
(-1,1.73)to[out= 0,in=180, looseness=1](0,1.73)
(4,1.73)to[out= 0,in=180, looseness=1](5,1.73)
(1,3.46)to[out= 120,in=-60, looseness=1] (0,5.19)
(3,0)to[out= -60,in=120, looseness=1] (4,-1.73);
\fill(2,1.73) circle (4pt);
\path[font=](3,-2.27) node[below] {$\mathcal{S}$};
\end{tikzpicture}\quad
\begin{tikzpicture}[scale=0.30]
\draw[color=black]
(2,-1.73)to[out= -120,in=60, looseness=1] (1,-3.46)
(1,0)to[out= 60,in=-120, looseness=1] (2,1.73)
(1,0)to[out=-60,in=120, looseness=1] (2,-1.73);
\draw[very thick, shift={(1,0)},scale=1,color=black,rotate=330]
(0,0)to[out= 90,in=-90, looseness=1](0,0.75)
(0,0.75)to[out= -45,in=135, looseness=1](0.15,0.6)
(0,0.75)to[out= -135,in=45, looseness=1](-0.15,0.6);
\draw[very thick, shift={(1,-3.46)},scale=1,color=black,rotate=330]
(0,0)to[out= 90,in=-90, looseness=1](0,0.75)
(0,0.75)to[out= -45,in=135, looseness=1](0.15,0.6)
(0,0.75)to[out= -135,in=45, looseness=1](-0.15,0.6);
\draw[very thick, shift={(2,-1.73)},scale=1,color=black,rotate=-330]
(0,0)to[out= 90,in=-90, looseness=1](0,0.75)
(0,0.75)to[out= -45,in=135, looseness=1](0.15,0.6)
(0,0.75)to[out= -135,in=45, looseness=1](-0.15,0.6);
\path[font=\footnotesize]
(2,-5.73) node[below] {The core of $\mathcal{S}$};
\end{tikzpicture}\qquad
\begin{tikzpicture}[scale=0.30]
\draw[color=black]
(-1,0)to[out= 120,in=-60, looseness=1] (-2,1.73)
(-1,0)to[out= -120,in=60, looseness=1] (-2,-1.73)
(-1,0)to[out= 0,in=180, looseness=1](1,0)
(2,1.73)to[out= 0,in=180, looseness=1](4,1.73)
(2,-1.73)to[out= 0,in=180, looseness=1](4,-1.73)
(2,1.73)to[out= 120,in=-60, looseness=1] (1,3.46)
(2,-1.73)to[out= -120,in=60, looseness=1] (1,-3.46)
(1,0)to[out= 60,in=-120, looseness=1] (2,1.73)
(1,0)to[out=-60,in=120, looseness=1] (2,-1.73);
\draw[color=black,dashed]
 (-2,1.73)to[out= 120,in=-60, looseness=1](-3,3.46)
 (-2,-1.73)to[out= -120,in=60, looseness=1](-3,-3.46)
(4,1.73)to[out= 0,in=180, looseness=1](5,1.73)
(4,-1.73)to[out= 0,in=180, looseness=1](5,-1.73)
(1,3.46)to[out= 120,in=-60, looseness=1] (0,5.19)
(1,-3.46)to[out= -120,in=60, looseness=1] (0,-5.19);
\fill(-1,0)circle (4pt);
\fill(1,0) circle (4pt);
\fill(2,1.73) circle (4pt);
\fill(2,-1.73) circle (4pt);
\path[font=](3,-4) node[below] {$\mathcal{N}_n$};
\draw[very thick, shift={(1,0)},scale=1,color=black,rotate=330]
(0,0)to[out= 90,in=-90, looseness=1](0,0.75)
(0,0.75)to[out= -45,in=135, looseness=1](0.15,0.6)
(0,0.75)to[out= -135,in=45, looseness=1](-0.15,0.6);
\draw[very thick, shift={(2,-1.73)},scale=1,color=black,rotate=-330]
(0,0)to[out= 90,in=-90, looseness=1](0,0.75)
(0,0.75)to[out= -45,in=135, looseness=1](0.15,0.6)
(0,0.75)to[out= -135,in=45, looseness=1](-0.15,0.6);
\end{tikzpicture}\quad
\begin{tikzpicture}[scale=0.30]
\draw[color=black]
(2,1.73)to[out= 0,in=180, looseness=1](4,1.73)
(2,1.73)to[out= 120,in=-60, looseness=1] (1,3.46)
(2,1.73)to[out= -120,in=60, looseness=1] (1,0);
\draw[color=black,dashed]
(4,1.73)to[out= 0,in=180, looseness=1](5,1.73)
(1,3.46)to[out= 120,in=-60, looseness=1] (0,5.19)
(1,0)to[out= -120,in=60, looseness=1] (0,-1.73);
\fill(2,1.73) circle (4pt);
\path[font=](3,-2.27) node[below] {$\mathcal{S}$};
\end{tikzpicture}\quad
\begin{tikzpicture}[scale=0.30]
\draw[color=black]
(1,0)to[out= 60,in=-120, looseness=1] (2,1.73)
(1,0)to[out=-60,in=120, looseness=1] (2,-1.73);
\draw[very thick, shift={(1,0)},scale=1,color=black,rotate=330]
(0,0)to[out= 90,in=-90, looseness=1](0,0.75)
(0,0.75)to[out= -45,in=135, looseness=1](0.15,0.6)
(0,0.75)to[out= -135,in=45, looseness=1](-0.15,0.6);
\draw[very thick, shift={(2,-1.73)},scale=1,color=black,rotate=-330]
(0,0)to[out= 90,in=-90, looseness=1](0,0.75)
(0,0.75)to[out= -45,in=135, looseness=1](0.15,0.6)
(0,0.75)to[out= -135,in=45, looseness=1](-0.15,0.6);
\path[font=\footnotesize]
(2,-3) node[below] {The longest}
(2,-4) node[below] {simple path}
(2,-5) node[below] {in the core of $\mathcal{S}$};
\end{tikzpicture}\qquad
\end{center}
\begin{caption}{The possible local structure of the graphs $G$, with relative networks 
$\mathcal{S}$ and cores, for $N=2, 3, 4$.\label{fig8}}
\end{caption}
\end{figure}
\noindent Hence, if $N=4$, in both two situations above there is in $\mathcal{S}$ at least one halfline with multiplicity two, thus such case is also excluded.\\
Then, we conclude that the only possible network with a core is when $N=2$ and $\mathcal{S}$ is given by two lines intersecting at the origin forming angles of $120/60$ degrees and the core consists of a collapsed segment.
\end{proof}

\begin{lem}\label{possibiliblowup}
Let $\mathcal{N}_0$ be a regular tree--like network and
let $\mathcal{N}_t$ be the maximal solution of the motion by curvature 
with fixed end--points in a smooth,  convex, bounded  
and open set $\Omega\subseteq\R^2$
in the maximal time interval $[0,T)$
with initial datum $\mathcal{N}_0$. 
Let $\mathcal{N}^{\lambda_i}_\tau$ be a sequence of parabolically rescaled curvature flows
around $(p_0,T)$ (see Definition~\ref{pararesc})
and suppose that $\textbf{M1}$ holds.
Then if $p_0\in\Omega\cap\mathcal{R}$,
$\mathcal{N}^{\lambda_i}_\tau$ converges (up to subsequences) to 
\begin{itemize}
\item a static flow of a straight line through the origin;
\item a static flow of a standard triod centered at the origin;
\item a static flow of a standard cross centered at the origin.
\end{itemize}
If $p_0\in\partial\Omega\cap\mathcal{R}$, then
$\mathcal{N}^{\lambda_i}_\tau$ converges (up to subsequences) to 
\begin{itemize}
\item a static flow of a halfline from the origin;
\item a static flow of two concurring halflines forming an angle of $120$ degrees
at the origin.
\end{itemize}
\end{lem}
\begin{proof}
By hypothesis we 
consider a sequence of parabolically rescaled curvature flows $\mathcal{N}^{\lambda_i}_\tau$ around $(p_0,T)$. 
Then, by Proposition~\ref{convergenza-a-shrinkers},
as $i\to\infty$, it converges to a 
self--similarly shrinking network flow $\mathcal{N}^\infty_\tau$, in $C^{1,\alpha}\loc \cap W^{2,2}\loc$, for almost all $\tau\in (-\infty, 0)$ and for any $\alpha \in (0,1/2)$.
Assuming that the conjecture $\textbf{M1}$ holds true, then $\mathcal{N}^\infty_\tau$ has unit multiplicity.

There is a one--to--one correspondence between shrinkers and self--similar shrinking flow:
if $\mathcal{S}=\{\eta^i\}_{i=1}^N$ is a shrinker, 
then the evolution given by  $\mathcal{N}_t=\{\gamma^i\}_{i=1}^N$ with
$\gamma^i(t,x)=\sqrt{-2t}\eta^i(x)$
is a self--similarly shrinking  flow in the time interval 
$(-\infty,0)$ with $\mathcal{N}_{-1/2}=\mathcal{S}$ and the other way around.
Thus to classify the tangent flows it is enough to look at the limit shrinkers.
By Lemma~\ref{lemmatree},
if we suppose that $p_0\not\in\partial\Omega$, then $\mathcal{N}^\infty_\tau$ can only be the ``static'' flow given by:
\begin{itemize}
\item a straight line;
\item a standard triod;
\item a standard cross.
\end{itemize}
If instead $p_0\in\partial\Omega$, 
by a reflection argument as before (see~\cite{ManNovPluSchu}, just before Theorem~14.4), we reduce ourselves again to the case $p_0\in\Omega$.
As a result 
the only two possibilities for $\mathcal{N}^\infty_\tau$ are the static flows given by:
\begin{itemize}
\item a halfline;
\item two concurring halflines forming an angle of $120$ degrees.
\end{itemize}
We thus have listed all possible cases.
\end{proof}

If we add an extra assumption of the lengths of the curves of the evolving network,
two of the previous cases are excluded. To be precise we have:

\begin{cor}\label{enanrem}
Let $\mathcal{N}_0$ be a regular tree--like network and
let $\mathcal{N}_t$ be the maximal solution of the motion by curvature 
with fixed end--points in a smooth, strictly  convex, bounded  
and open set $\Omega\subseteq\R^2$
in the maximal time interval $[0,T)$
with initial datum $\mathcal{N}_0$. 
Assume that the lengths $L^i(t)$ of all the curves of the networks satisfy
\begin{equation}\label{Lbasso}
\lim_{t\to T}\frac{L^i(t)}{\sqrt{T-t}}=+\infty\,.
\end{equation}
Let $\mathcal{N}^{\lambda_i}_\tau$ be a sequence of parabolically rescaled curvature flows
around $(p_0,T)$ (see Definition~\ref{pararesc}) and suppose that $\textbf{M1}$ holds.
Then if the rescaling point belongs to $\Omega\cap\mathcal{R}$, 
$\mathcal{N}^{\lambda_i}_\tau$ converges (up to subsequences)
 to:
\begin{itemize}
\item a static flow of a straight line through the origin (in this case $\widehat\Theta(p_0)=1$);
\item a static flow of a standard triod centered at the origin (in this case $\widehat\Theta(p_0)=3/2$).
\end{itemize}
If the rescaling point $p_0$ is a fixed end--point of the evolving network
(on the  boundary of $\Omega$), then 
$\mathcal{N}^{\lambda_i}_\tau$ converges (up to subsequences)
 to:
\begin{itemize}
\item a static flow of a halfline from the origin(in this case $\widehat\Theta(p_0)=1/2$).
\end{itemize}
If  the rescaling point does not belongs to $\mathcal{R}$, then the limit is the empty set
and $\widehat{\Theta}(p_0)=0$.
\end{cor}
\begin{proof}
The possible limit flow is among the ones listed in the previous lemma.
Moreover the bound from below on the lengths
exclude the possibility that the shrinker has a core.
Clearly the only remaining options are the line, the standard triod and the halfline.
\end{proof}

\subsection{Local regularity}

\begin{defn}\label{ublr}
We say that a network
$\mathcal{N}$ has  bounded length ratios by the constant $C>0$, if 
\begin{equation}\label{blr}
\mathcal{H}^1(\mathcal{N}\cap B_R(\overline{p}))\leq C R\,, 
\end{equation}
for every $\overline{p}\in\R^2$ and $R>0$.
\end{defn}

\begin{thm}\label{thm:locreg.2}
Let $\mathcal{N}_t$ for ${t\in (T_0,T)}$ be a curvature flow of a smooth, regular 
network
with uniformly bounded length ratios by some constant $L$. Let $(p_0,t_0)\in\R^2\times(T_0,T)$ such that $p_0\in\mathcal{N}_{t_0}$.
Let $\varepsilon, \eta>0$. If
\begin{equation}\label{eq:locreg.0.5}
{\Theta}_{p,t}(t-r^2) \leq \Theta_{\SS^1} -\varepsilon\,,
\end{equation}
for all $(p,t) \in B_\rho(p_0)\times (t_0-\rho^2,t_0)$ and
$0<r<\eta\rho$, for some $\rho >0$, where $T_0+(1+\eta)\rho^2\leq
t_0<T$, then there exists $K=K(\varepsilon,\eta,C)$ such that
$$
k^2(p,t) \leq \frac{K}{\sigma^2\rho^2}\,,
$$
for all $\sigma \in (0,1)$ and every $(p,t)$ with  $p\in\mathcal{N}_t\cap B_{(1-\sigma)\rho}(p_0)$
and $t\in(t_0-(1-\sigma)^2\rho^2,t_0)$. 
\end{thm}
\begin{proof}
See~\cite[Theorem 1.3]{IlNevSchu}.
\end{proof}

The following corollary is then an extension of White's result~\cite[Theorem~3.5]{Wh} to the curvature flow $\mathcal{N}_t$ of a network in $\Omega\subseteq\R^2$ with fixed end--points on $\partial\Omega$.

\begin{cor}\label{regcol}
Let $\mathcal{N}_t$ be a maximal solution to the motion by curvature in the maximal time 
interval $[0,T)$
in $\overline{\Omega}$, possibly with end--points $P^1,\ldots,P^\ell	\in\partial\Omega$.
Suppose that
\begin{itemize}
\item either at $p_0\in\Omega$ there holds $\widehat{\Theta}(p_0,T)\leq 3/2$;
\item or at $p_0\in\partial\Omega$ there holds $\widehat{\Theta}(p_0,T)\leq 1/2$.
\end{itemize}
Then the curvature is uniformly bounded along the flow  in a neighborhood of $p_0$, for $t\in[0,T)$.
\end{cor}
\begin{proof}
It is well known, that in these hypetheses, the family of networks $\mathcal{N}_t$ has uniformly bounded length ratios (this is due to Stone~\cite{stone1} -- see also~\cite[Lemma~8.15]{ManNovPluSchu}). Then, as $\widehat{\Theta}(p_0)=\widehat{\Theta}(p_0,T)\leq3/2$, by the monotonicity and the properties of $\Theta_{p_0,T}(t)$ (see~\cite[Section 7]{ManNovPluSchu} for details), there exists $\rho_1\in(0,1)$ such that $\Theta_{p_0,T}(T-\rho_1^2)<3/2+\delta/2$, for some small $\delta>0$. The function $(p,t)\mapsto\Theta_{p,t}(t-\rho_1^2)$ is continuous, hence, we can find $\rho<\rho_1$ such that if $(p,t)\in B_\rho(p_0)\times(T-\rho^2,T)$, then $\Theta_{p,t}(t-\rho^2_1)<3/2+\delta$, thus, again by the monotonicity (possibly choosing smaller $\rho_1$ and $\rho$), also $\Theta_{p,t}(t-r^2)<3/2+\delta$, for any $r\in(0,\rho/2)$, as clearly $(t-r^2)>( t-\rho_1^2)$. 

This implies that if $3/2+\delta<\Theta_{\SS^1}=\sqrt{2\pi/e}$, for any $t_0$ close enough to $T$, the hypotheses of Theorem~\ref{thm:locreg.2} (see the first point of Remark~9.4 in~\cite{ManNovPluSchu}) are satisfied at $(p_0,t_0)$, for $\eta=3/4$ and $\varepsilon=\Theta_{\SS^1}-3/2-\delta>0$. Choosing $\sigma=1/2$, we conclude that 
$$
k^2(p,t) \leq \frac{4C(\varepsilon,3/4)}{\rho^2}
$$
for every $(p,t)$ such that $t\in(t_0-\rho^2/4,t_0)$ and $p\in\mathcal{N}_t\cap B_{\rho/2}(p_0)$. Since this estimate on the curvature is independent of $t_0<T$, it must hold for every $t\in(T-\rho^2/4,T)$ and $p\in\mathcal{N}_t\cap B_{\rho/2}(p_0)$ and we are done.
\end{proof}

\section{Estimates on geometric quantities}

We list here several relations fulfilled by geometric quantities at  
the triple junctions of the solutions of smooth flow of regular networks.
However before stating the lemma we list the evolution equation of several quantities
used in the sequel.

If $\gamma$ is a curve moving by
$\gamma_t=k\nu+\zeta\tau\,,
$
then the following commutation rule holds
\begin{equation}\label{commut}
\dert\ders=\ders\dert +
(k^2 -\zeta_s)\ders\,.
\end{equation}
To show it we let
 $f:[0,1]\times[0,T)\to\R$ be a smooth function, then
\begin{align*}
\dert\ders f - \ders\dert f =&\, \frac{f_{tx}}{\vert\gamma_x\vert} -
\frac{\langle \gamma_x\,\vert\,\gamma_{xt}\rangle f_x}
{\vert\gamma_x\vert^3}
- \frac{f_{tx}}{\vert\gamma_x\vert} = - {\langle
  \tau\,\vert\,\partial_s\gamma_t\rangle}\partial_sf\\
=&\, - {\langle\tau\,\vert\,\partial_s(\zeta\tau+k\nu)\rangle}\partial_sf=
 (k^2 - \zeta_s)\ders f\,.
\end{align*}
Then, thanks to the commutation rule for an evolving curve we can compute
\begin{align}
\dert\mathrm{d}s=&\left(\partial_s\zeta-k^2\right)\mathrm{d}s\nonumber\\
\dert\tau=&\,
\dert\ders\gamma=\ders\dert\gamma+(k^2-\zeta_s)\ders\gamma =
\ders(\zeta\tau+k\nu)+(k^2-\zeta_s)\tau =
(k_s+k\zeta)\nu\label{derttau}\,,\\
\dert\nu=&\, \dert({\mathrm R}\tau)={\mathrm
R}\,\dert\tau=-(k_s+k\zeta)\tau\label{dertdinu}\,,\\
\dert k=&\, \dert\langle \ders\tau\,|\, \nu\rangle=
\langle\dert\ders\tau\,|\, \nu\rangle\label{dertdik}
= \langle\ders\dert\tau\,|\, \nu\rangle +
(k^2-\zeta_s)\langle\ders\tau\,|\, \nu\rangle\\
=&\, \ders\langle\dert\tau\,|\, \nu\rangle + k^3-k\zeta_s =
\ders(k_s+k\zeta) + k^3-k\zeta_s\nonumber\\
=&\, k_{ss}+k_s\zeta + k^3\nonumber\,,
\end{align}
where $\mathrm{R}$ is the anticlockwise rotation of $\pi/2$.

\begin{lem}\label{condizioni-giunzione}
Let $\mathcal{N}_t$ be a smooth flow of regular networks in $[0,T)$ and 
let $\gamma^1$, $\gamma^2$ and $\gamma^3$ be tree curves concurring at a
triple junction $\mathcal{O}$.
Then for every $t\in [0,T)$ at each triple junction $\mathcal{O}$ the following relations
hold true:
\begin{align}
\sqrt{3}\zeta^{i}&=\kappa^{i-1}-\kappa^{i+1}\,,\label{zetafunzdik}\\
\sum_{i=1}^3 \kappa^{i}=\sum_{i=1}^3 \zeta^{i}&=0\,,\label{sommakzero}\\
\partial_s\kappa^i + \zeta^i\kappa^i&=\partial_s\kappa^j + \zeta^j\kappa^j\,,\label{derivatatau}\\
\sum_{i=1}^3 \kappa^{i}\partial_s\kappa^{i}+\zeta^{i}|\kappa^{i}|^2&=0\,,\label{allagiunzione}\\
3\sum_{i=1}^3\zeta^i \kappa^i \partial_t\kappa^i &= \partial_t \sum_{i=1}^3\zeta^i \left(\kappa^i\right) ^2\,.\label{lowerboundaryterm}
\end{align}
with the convention that the superscripts appearing in sums are to be 
considered ``modulus $3$''.
\end{lem}
\begin{proof}
Let $\gamma^1$, $\gamma^2$ and $\gamma^3$ be tree curves concurring at a
triple junction $\mathcal{O}$, and suppose without loss of generality that
$\gamma^1(t,0)=\gamma^2(t,0)=\gamma^3(t,0)=\mathcal{O}$
for every $t\in [0,T)$ (for simplicity in the following we omit the dependence on $(t,0)$).
Differentiating in time this concurrency condition we obtain that the velocity of the three curves
at $\mathcal{O}$ are the same, that is 
$$
\zeta^{i}\tau^{i}+\kappa^{i}\nu^{i}=
\zeta^{j}\tau^{j}+\kappa^{j}\nu^{j}\,,
$$
for every $i,j\in\{1,2,3\}$.\\
Multiplying these vector identities
by $\tau^{l}$ and $\nu^{l}$ and varying $i, j, l$,  thanks to the
conditions $\sum_{i=1}^{3}\tau^{i}=\sum_{i=1}^{3}\nu^{i}=0$, we get
the relations
\begin{gather*}
\zeta^{i}=-\zeta^{i+1}/2-\sqrt{3}\kappa^{i+1}/2\,,\\
\zeta^{i}=-\zeta^{i-1}/2+\sqrt{3}\kappa^{i-1}/2\,,\\
\kappa^{i}=-\kappa^{i+1}/2+\sqrt{3}\zeta^{i+1}/2\,,\\
\kappa^{i}=-\kappa^{i-1}/2-\sqrt{3}\zeta^{i-1}/2\,,
\end{gather*}
that leads to
\begin{gather*}
\zeta^{i}=\frac{\kappa^{i-1}-\kappa^{i+1}}{\sqrt{3}}\,,\\
\kappa^{i}=\frac{\zeta^{i+1}-\zeta^{i-1}}{\sqrt{3}}\,,
\end{gather*}
which implies
\begin{equation}
\sum_{i=1}^3 \kappa^{i}=\sum_{i=1}^3\zeta^{i}=0\,,
\end{equation}
and so we obtained both~\eqref{zetafunzdik} and~\eqref{sommakzero}.
Now we differentiate in time the angle condition $\sum_{i=1}^3\tau^i=0$, getting
$\sum_{i=1}^3(\partial_s\kappa^i + \zeta^i\kappa^i)\nu^i=0$.
The fact that also $\sum_{i=1}^3\nu^i=0$ implies
\begin{equation*}
\partial_s\kappa^i + \zeta^i\kappa^i=\partial_s\kappa^j + \zeta^j\kappa^j\,.
\end{equation*}
Thus
\begin{equation}
0=(\partial_s\kappa^1 + \zeta^1\kappa^1)\sum_{i=1}^3 \kappa^{i}
=\sum_{i=1}^3 \kappa^{i}\partial_s\kappa^{i}+\zeta^{i}|\kappa^{i}|^2\,.
\end{equation}
Now we show equation~\eqref{lowerboundaryterm}.
Keeping in mind~\eqref{zetafunzdik} and~\eqref{sommakzero}, we obtain
\begin{align*}
\sqrt{3}\sum_{i=1}^3\zeta^i \kappa^i \partial_t\kappa^i &= \sum_{i=1}^3 \left(\kappa^{i-1}-\kappa^{i+1}\right) \kappa^i\partial_t\kappa^i \\
&=\sum_{i=1}^3 \left[\left(\kappa^{i+1}\right)^2-\left(\kappa^{i-1}\right)^2\right]\partial_t\kappa^i\,,
\end{align*}
and
\begin{align*}
\sqrt{3} \partial_t \sum_{i=1}^3\zeta^i \left( \kappa^i\right) ^2
&=\sqrt{3}\sum_{i=1}^3\left(\partial_t\zeta^i\left( \kappa^i\right)^2+2\zeta^i \kappa^i\partial_t\kappa^i\right)\\
&=\sum_{i=1}^3 \left(\partial_t\kappa^{i-1}-\partial_t\kappa^{i+1}\right)\left(\kappa^i\right)^2
  +2\sum_{i=1}^3\left(\kappa^{i-1}-\kappa^{i+1}\right)\kappa^i\partial_t\kappa^i\\
&=\sum_{i=1}^3\left[\left(\kappa^{i+1}\right)^2-\left(\kappa^{i-1}\right)^2
+2\kappa^ik^{i-1}-2\kappa^ik^{i+1}\right]\partial_t\kappa^i\\
&=\sum_{i=1}^3\left[\left(\kappa^{i+1}\right)^2-\left(\kappa^{i-1}\right)^2
-2(\kappa^{i-1}+\kappa^{i+1})\kappa^{i-1}+2(\kappa^{i-1}+\kappa^{i+1})\kappa^{i+1}\right]\partial_t\kappa^i\\
&=3\sum_{i=1}^3\left[\left(\kappa^{i+1}\right)^2-\left(\kappa^{i-1}\right)^2
\right]\partial_t\kappa^i\,,
\end{align*}
and so equality~\eqref{lowerboundaryterm} is proved.

\end{proof}

\begin{lem}\label{kinfty}
Let $\mathcal{N}_0$ be an initial regular tree--shaped network 
with end--points $P^1, \dots, P^\ell\in\partial\Omega$ (not necessary fixed). 
Suppose that  $\mathcal{N}_t$ is a solution to the motion by curvature 
in $[0,T)$ with initial datum  $\mathcal{N}_0$ such that the square of the curvature at the end--points of $\mathcal{N}_t$ is uniformly bounded in time by some constant $C\leq 0$.
Suppose that for every point $P\in\mathcal{N}_0$ there is a path 
from $P$ to an end--point on $\partial\Omega$ composed of at most 
 $n$ curves  of $\mathcal{N}_0$.
Then there exists a constant $D(n)$ depending only on $n$ such that it holds
\begin{equation}\label{inkappa}
\Vert\kappa\Vert_{L^\infty}^2 \leq   4^{n-1}C+D(n)\Vert\kappa\Vert_{L^2}\Vert\partial_s\kappa\Vert_{L^2}\,.
\end{equation}
\end{lem}
\begin{proof}
Let us start from the simple case of a network $\mathcal{N}_0$ composed of five curves, 
two triple junctions $\mathcal{O}^1, \mathcal{O}^2$ and
four end--points $P^1, P^2, P^3, P^4$ on $\partial\Omega$. 
In this case $n=2$.
We call $\gamma^i$, for $i\in\{1,2,3,4\}$, the curve connecting $P^i$ with one of the two triple junctions and $\gamma^5$ the curve between the two triple junctions (see Figure~\ref{tree}).
\begin{figure}[H]
\begin{center}
\begin{tikzpicture}[scale=1]
\draw
 (-3.73,0) 
to[out= 50,in=180, looseness=1] (-2.3,0.7) 
to[out= 60,in=180, looseness=1.5] (-0.45,1.55) 
(-2.3,0.7)
to[out= -60,in=130, looseness=0.9] (-1,-0.3)
to[out= 10,in=100, looseness=0.9](0.1,-0.8)
(-1,-0.3)
to[out=-110,in=50, looseness=0.9](-2.7,-1.7);
\draw[color=black,scale=1,domain=-3.141: 3.141,
smooth,variable=\t,shift={(-1.72,0)},rotate=0]plot({2.*sin(\t r)},
{2.*cos(\t r)});
\path[font=\small]
 (-3.73,0) node[left]{$P^2$}
 (-2.9,-1.7)node[below]{$P^3$}
 (0.1,-0.8)node[right]{$P^4$}  
  (-0.40,1.6) node[right]{$P^1$}
   (-3,0.6) node[below] {$\gamma^2$}
   (-1.5,1.3) node[right] {$\gamma^1$}
   (-1.1,-1.2)[left] node{$\gamma^3$}
   (0,-0.8)[left] node{$\gamma^4$}
    (-1.3,0.5)[left] node{$\gamma^5$}
  (-2.45,1.3) node[below] {$\mathcal{O}^1$}
   (-1.4,-0.1) node[below] {$\mathcal{O}^2$};  
\end{tikzpicture}
\end{center}
\begin{caption}{A tree--like network with five curves.\label{tree}}
\end{caption}
\end{figure}
Fixed a time $t\in[0,T)$, let $P\in\gamma^i\subseteq\mathcal{N}_t$ for a certain  $i\in\{1,2,3,4\}$. Then
\begin{equation*}
[\kappa^i(P)]^2=[\kappa^i(P^i)]^2+ 2 \int_{P^i}^P \kappa \kappa_s\,\mathrm{d}s 
\leq C + 2 \Vert \kappa\Vert_{L^2} \Vert\partial_s\kappa\Vert_{L^2}\,,
\end{equation*}
hence, for every $P\in\mathcal{N}_t\setminus\gamma^5$ we have
\begin{equation*}
[\kappa^i(P)]^2\leq   C + 2 \Vert\kappa\Vert_{L^2} \Vert\partial_s\kappa\Vert_{L^2}\,.
\end{equation*}
Assume now instead that $P\in \gamma^5$. 
By the previous argument for all $i\in\{1, 2, 3, 4\}$ we have
\begin{align*}
[\kappa^i(\mathcal{O}^1)]^2\leq  C + 2 \Vert \kappa\Vert_{L^2} \Vert\partial_s\kappa\|_{L^2}\,,\\
[\kappa^i(\mathcal{O}^2)]^2\leq  C + 2 \Vert\kappa\Vert_{L^2}\Vert\partial_s\kappa\|_{L^2}\,.
\end{align*}
Moreover by~\eqref{sommakzero} we have $\kappa^5(\mathcal{O}^1)=-\kappa^1(\mathcal{O}^1)-\kappa^2(\mathcal{O}^1)$
and $\kappa^5(\mathcal{O}^2)=-\kappa^3(\mathcal{O}^2)-\kappa^4(\mathcal{O}^2)$,
hence
\begin{align*}
[\kappa^5(\mathcal{O}^1)]^2\leq 4C + 8\Vert \kappa\Vert_{L^2} \Vert\partial_s\kappa\Vert_{L^2}\,,\\
[\kappa^5(\mathcal{O}^2)]^2\leq 4C + 8\Vert \kappa\Vert_{L^2} \Vert\partial_s\kappa\Vert_{L^2}\,.
\end{align*}
Arguing as before, we get
\begin{equation*}
[\kappa^5(P)]^2 = [\kappa^5(\mathcal{O}^1)]^2 + 2 \int_{\mathcal{O}^1}^P\kappa\partial_s \kappa\, \mathrm{d}s 
\leq 4C + 8\Vert\kappa\Vert_{L^2} \Vert\partial_s\kappa\Vert_{L^2} + 2 \int_{\mathcal{O}^1}^P \kappa\partial_s \kappa\, \mathrm{d}s\,.
\end{equation*}
In conclusion, we get the uniform in time inequality for $\mathcal{N}_t$
\begin{equation*}
\Vert \kappa\Vert_{L^\infty}^2 \le   4C + 10\Vert\kappa\Vert_{L^2} \Vert\partial_s\kappa\Vert_{L^2}.
\end{equation*}
In the general case, since $\mathcal{N}_t$ are all trees homeomorphic to $\mathcal{N}_0$, we can argue similarly to get the conclusion by induction on $n$.
\end{proof}

\begin{lem}\label{kappa2}
Let $\Omega\subseteq\mathbb{R}^2$ be open, convex and regular, 
let $\mathcal{N}_0$ be a tree with end--points $P^1,\dots, P^\ell$ on $\partial\Omega$
and let $\mathcal{N}_t$  be a smooth evolution by curvature in $[0,T)$
with initial datum  $\mathcal{N}_0$.
Suppose that during the evolution the end--points $P^1,\dots, P^\ell$
either are fixed or there 
exist uniform (in time) constants $C_j$, for every $j\in\NN$, such that 
\begin{equation}\label{endsmooth}
\vert\partial_s^j\kappa(P^r,t)\vert+\vert\partial_s^j\zeta(P^r,t)\vert\leq C_j\,,
\end{equation}
for every $t\in[0,T)$ and $r\in\{1,\dots,\ell\}$.
Then $\Vert \kappa\Vert_{L^2}^2$ is uniformly bounded on $[0,\widetilde{T})$ by $\sqrt{2}\bigl[\Vert \kappa(\cdot,0)\Vert_{L^2}^2+1\bigr]$, where 
$$
\widetilde{T}=\min\,\Bigl\{ T, 1\big/
8C\,\bigl(\Vert \kappa(\cdot,0)\Vert_{L^2}^2+1\bigr)^2\Bigr\}\,.
$$
Here the constant $C$ depends only on the number $n\in\NN$ of Lemma~\ref{kinfty} and the constants $C_j$. 
\end{lem}
\begin{proof}
By direct computations and the relation~\eqref{allagiunzione} we have
\begin{align}
\partial_t\int_{\mathcal{N}_t} \kappa^2\,\mathrm{d}s 
=&\int_{\mathcal{N}_t} 2\kappa\partial_t\kappa+\kappa^2(\partial_s\zeta-\kappa^2)\,\mathrm{d}s \\
=&\int_{\mathcal{N}_t} 2\kappa\left(\partial_s^2\kappa+\partial_s\kappa\zeta+\kappa^3\right)
+\kappa^2(\partial_s\zeta-\kappa^2)\,\mathrm{d}s \\
=& \int_{\mathcal{N}_t}-2(\partial_s\kappa)^2+\kappa^4\,\mathrm{d}s
+ \sum_{r=1}^\ell 2\kappa\partial_s\kappa+\zeta\kappa^2\,\biggr\vert_{\text{{ at    the end--point $P^r$}}}\\
-& \sum_{p=1}^m\sum_{i=1}^3 2\kappa^{pi}\partial_s\kappa^{pi}+\zeta^{pi}(\kappa^{pi})^2\,\biggr\vert_{\text{{ at    the $3$--point $O^p$}}}\\
\leq&\, -2\int_{\mathcal{N}_t} \partial_s\kappa^2\, \mathrm{d}s + \int_{\mathcal{N}_t} k^4\, \mathrm{d}s\nonumber
+ \sum_{p=1}^m\sum_{i=1}^3\zeta^{pi}\left(  k^{pi}\right) ^2\,\biggr\vert_{\text{{ at    the $3$--point $O^p$}}}+\ell C_0C_{1}\nonumber\\
\le&\,-2\int_{\mathcal{N}_t} \partial_s\kappa^2\,\mathrm{d}s + \|k\|_{L^\infty}^2\int_{\mathcal{N}_t}k^2\,\mathrm{d}s
+ C \|\kappa\|_{L^\infty}^3+C\,.\label{kappat}
\end{align}
By estimate~\eqref{inkappa} and the Young inequality, we then obtain
\begin{align*}
\Vert\kappa\Vert_{L^\infty}^3 
&\leq C_n + C_n \Vert\kappa\Vert_{L^2}^\frac 32 \Vert\partial_s\kappa\Vert_{L^2}^\frac 32 \le C_n+ \eps \Vert\partial_s\kappa\Vert_{L^2}^2 + C_{n,\eps} \Vert\kappa\Vert_{L^2}^6\,,\\
\Vert\kappa\Vert_{L^\infty}^2\Vert\kappa\Vert_{L^2}^2 &\le C_n \Vert\kappa\Vert_{L^2}^2 + D_n\Vert\kappa\Vert_{L^2}^3 \Vert\partial_s\kappa\Vert_{L^2} \le C_n \Vert\kappa\Vert_{L^2}^2+\eps \Vert\partial_s\kappa\Vert_{L^2}^2 + C_{n,\eps} \Vert\kappa\|_{L^2}^6\,,
\end{align*}
for every small $\varepsilon>0$ and a suitable constant 
$C_{n,\varepsilon}$.\\
Plugging these estimates into inequality~\eqref{kappat} we get
\begin{align}
\partial_t\int_{\mathcal{N}_t} k^2 \mathrm{d}s 
\leq &\,-2\Vert \partial_sk\Vert^2+ \|k\|_{L^\infty}^2\Vert k\Vert^2+ C \|k\|_{L^\infty}^3+C\nonumber\\
\leq &\,-2\Vert \partial_s\kappa\Vert^2+ C_n \|k\|_{L^2}^2+
 \eps \|\partial_s\kappa\|_{L^2}^2 + C_{n,\eps} \|k\|_{L^2}^6
+C_n+ \eps \|\partial_s\kappa\|_{L^2}^2 + C_{n,\eps} \|k\|_{L^2}^6+C_n\nonumber\\
\leq&\, C\Bigl(  \int_{\mathcal{N}_t} k^2 \mathrm{d}s\Bigr)^3+C\,,\label{kappatt}
\end{align}
Where we chose $\eps=1/2$ and the constant $C$ depends only on the number $n\in\NN$ of Lemma~\ref{kinfty} and the constants in condition~\eqref{endsmooth}. 

Calling $y(t)= \int_{\mathcal{N}_t} k^2\, \mathrm{d}s+1$, we can rewrite inequality~\eqref{kappatt} as the differential ODE 
$$
y'(t)\leq 2Cy^3(t)\,,
$$
hence, after integration, we get
$$
y(t)\leq \frac{1}{\sqrt{\frac{1}{y^2(0)}-4Ct}}
$$
and, choosing $\widetilde{T}$ as in the statement, the conclusion is straightforward.
\end{proof}

We prove now that if we consider a tree--like network composed of five curves
with two triple junctions and four end--points on $\partial\Omega$, which move in a ``controlled way'', as in the previous lemma, such that the boundary curves do not vanish during the evolution, then the $L^2$--norm of $\partial_s\kappa$ is bounded until $\Vert k\Vert_{L^2}$ stays bounded.
This is the crucial local estimate needed in the proof of Theorem~\ref{cross}.

\begin{lem}\label{kappaesse2}
Let $\Omega\subseteq\mathbb{R}^2$ be open, convex and regular, 
let $\mathcal{N}_0$ be a tree with five curves, 
two triple junctions $\mathcal{O}^1, \mathcal{O}^2$ and
four end--points $P^1, P^2, P^3, P^4$ on $\partial\Omega$, as in Figure~\ref{tree}, satisfying assumption~\eqref{endsmooth} and assume that $\mathcal{N}_t$, for $t\in [0,T)$, is a smooth evolution by curvature of the network $\mathcal{N}_0$ such that $\Vert \kappa\Vert_{L^2}$ is uniformly bounded on 
$[0,T)$.\\
If the lengths of the curves of the network arriving at the end--points are uniformly bounded below by some constant $L>0$, then $\Vert \partial_s\kappa\Vert_{L^2}$ is uniformly bounded on $[0,T)$.
\end{lem}

\begin{proof}
We first estimate $\Vert \partial_s\kappa\Vert^2_{L^\infty}$ in terms of $\|\partial_s\kappa\|_{L^2}$ and $\|\partial_s^2\kappa\|_{L^2}$.\\
Fixed a time $t\in[0,T)$, let $Q\in\gamma^i\subseteq\mathcal{N}_t$, for some $i\le 4$. We compute
\begin{eqnarray*}
[\partial_s\kappa^i(Q)]^2 &=& [\partial_s\kappa^i(P^i)]^2 + 2 \int_{P^i}^Q \partial_s\kappa \partial_s^2\kappa\,\mathrm{d}s \le C + 2 \|\partial_s\kappa\|_{L^2} \|\partial_s^2\kappa\|_{L^2}\,,
\end{eqnarray*}
hence, in this case, 
\[
[\partial_s\kappa^i(Q)]^2 \le   C + 2 \|\partial_s\kappa\|_{L^2} \|\partial_s^2\kappa\|_{L^2}\,,
\]
for every $Q\in\mathcal{N}_t\setminus\gamma^5$.\\
Assume now instead that $Q\in \gamma^5$. Recalling~\eqref{derivatatau}, we get 
$$
\partial_s\kappa^5(\mathcal{O}^1)=\partial_s\kappa^i(\mathcal{O}^1)
+\zeta^i(\mathcal{O}^1)\kappa^i(\mathcal{O}^1)-\zeta^5(\mathcal{O}^1)\kappa^5(\mathcal{O}^1)\,,
$$
hence,
\begin{align*}
\vert \partial_s\kappa^5(\mathcal{O}^1)\vert&\leq \vert \partial_s\kappa^i(\mathcal{O}^1)\vert +C\Vert \kappa\Vert^2_{L^\infty}\\
&\leq \vert \partial_s\kappa^i(\mathcal{O}^1)\vert +C\Vert \kappa\Vert_{L^2}\Vert \partial_s\kappa\Vert_{L^2}+C\\
&\leq \vert \partial_s\kappa^i(\mathcal{O}^1)\vert +C\left(1+\Vert \partial_s\kappa\Vert_{L^2}\right)\,,
\end{align*}
by Lemma~\ref{kappa2}. Then,
\begin{equation}\label{kesse5}
[\partial_s\kappa^5(\mathcal{O}^1)]^2\leq 2[\partial_s\kappa^i(\mathcal{O}^1)]^2+C\left(1+\Vert \partial_s\kappa\Vert^2_{L^2}\right)
\end{equation}
and it follows
\begin{eqnarray*}
[\partial_s\kappa^5(Q)]^2 &=& [\partial_s\kappa^5(\mathcal{O}^1)]^2 + 2 \int_{\mathcal{O}^1}^Q \partial_s\kappa \partial_s^2\kappa\,\mathrm{d}s\\ 
&\le& 2[\partial_s\kappa^i(\mathcal{O}^1)]^2 + C\left(1+\Vert \partial_s\kappa\Vert^2_{L^2}\right)+ 2 \int_{\mathcal{O}^1}^Q \partial_s\kappa \partial_s^2\kappa\,\mathrm{d}s\\
&\le& C + C\Vert \partial_s\kappa\Vert^2_{L^2}+2 \|\partial_s\kappa\|_{L^2} \|\partial_s^2\kappa\|_{L^2}\,,
\end{eqnarray*}
since, by the previous argument, we have $[\kappa^i_s(\mathcal{O}^1)]^2, [\kappa^i_s(\mathcal{O}^2)]^2\leq  C + 2 \|\partial_s\kappa\|_{L^2} \|\partial_s^2\kappa\|_{L^2}$, for all $i\in\{1, 2, 3, 4\}$. Hence, we conclude
\begin{equation}
\Vert \partial_s\kappa\Vert_{L^\infty}^2\leq C + C\Vert \partial_s\kappa\Vert^2_{L^2}+2 \|\partial_s\kappa\|_{L^2} \|\partial_s^2\kappa\|_{L^2}\,.
\end{equation}

We now pass to estimate $\|\partial_s\kappa\|_{L^2}$:

\begin{align}
\partial_t\int_{\mathcal{N}_t} (\partial_s\kappa)^2\,\mathrm{d}s 
=&\int_{\mathcal{N}_t} 2\partial_s\kappa\partial_t\partial_s\kappa+(\partial_s\kappa)^2(\partial_s\zeta-\kappa^2)\,\mathrm{d}s \\
=&\int_{\mathcal{N}_t} 2\partial_s\kappa
\left(\partial_s^3\kappa+\partial_s^2\kappa\zeta+4\kappa^2\partial_s\kappa\right)
+(\partial_s\kappa)^2(\partial_s\zeta-\kappa^2)\,\mathrm{d}s \\
=& \int_{\mathcal{N}_t}-2(\partial_s^2\kappa)^2+7\kappa^2(\partial_s\kappa)^2\,\mathrm{d}s
+ \sum_{r=1}^\ell 2\partial_s\kappa\partial_s^2\kappa+\zeta(\partial_s\kappa)^2\,\biggr\vert_{\text{{ at    the end--point $P^r$}}}\\
-& \sum_{p=1}^m\sum_{i=1}^3 2\partial_s\kappa^{pi}\partial_s^2\kappa^{pi}+\zeta^{pi}(\partial_s\kappa^{pi})^2\,\biggr\vert_{\text{{ at    the $3$--point $O^p$}}}\\
\leq&
-2\int_{\mathcal{N}_t} \partial_s^2\kappa^2\,\mathrm{d}s + 7\int_{\mathcal{N}_t} k^2\partial_s\kappa^2\,\mathrm{d}s+\ell C_{1}C_2\\
-&\sum_{p=1}^2\sum_{i=1}^3 2\partial_s k^{pi} \partial_s^2k^{pi}+\zeta ^{pi} \left( \partial_s k^{pi}\right) ^2\,
\biggr\vert_{\text{{ at    the $3$--point $O^p$}}}\,.\label{kappast}
\end{align}
Now we use~\eqref{lowerboundaryterm} to lower the differentiation order of  
the term $\sum_{i=1}^3\partial_s\kappa^i \partial_s^2\kappa^i$. From the formula $\dert k=\partial_s^2\kappa+\partial_s\kappa\zeta + k^3$ and from the fact that $\sum_{i=1}^3\partial_t\kappa^i=\partial_t\sum_{i=1}^3\kappa^i=0$, we get 
\begin{eqnarray}
\sum_{i=1}^3\partial_s\kappa^i \partial_s^2\kappa^i 
&=& \sum_{i=1}^3\partial_s\kappa^i\bigl[\partial_t\kappa_{t}-\zeta^i  \partial_s\kappa^i-\left( \kappa^i\right) ^3\bigr]\nonumber
\\
&=& \sum_{i=1}^3\bigl(\partial_s\kappa^i+\zeta^i \kappa^i-\zeta^i \kappa^i\bigr)\partial_t\kappa^i
-\sum_{i=1}^3\zeta^i \left( \partial_s\kappa^i\right)^2 + \left( \kappa^i\right)^3 \partial_s\kappa^i\nonumber
\\
&=& \sum_{i=1}^3\bigl(\partial_s\kappa^i+\zeta^i \kappa^i\bigr)\partial_t\kappa^i
-\sum_{i=1}^3\zeta^i \kappa^i \partial_t\kappa^i
-\sum_{i=1}^3\left[\zeta^i \left( \partial_s\kappa^i\right)^2 + \left( \kappa^i\right)^3 \partial_s\kappa^i\right]\nonumber
\\
&=& - \partial_t \sum_{i=1}^3\zeta^i \left( \kappa^i\right) ^2\bigr/3 
-\sum_{i=1}^3\left[\zeta^i \left( \partial_s\kappa^i\right)^2 + \left( \kappa^i\right) ^3 \partial_s\kappa^i\right]\,,\label{kessekesseesse}
\end{eqnarray}
at the triple junctions $\mathcal{O}^1$ and $\mathcal{O}^2$, where we used the fact that $\partial_s\kappa^i+\zeta^i \kappa^i$ is independent of $i\in\{1, 2, 3\}$.\\
Substituting this equality into estimate~\eqref{kappast}, we obtain
\begin{align}
\partial_t \int_{\mathcal{N}_t} (\partial_s\kappa)^2\,\mathrm{d}s
\le&\, -2\int_{\mathcal{N}_t} \partial_s^2\kappa^2\,\mathrm{d}s + 7\int_{\mathcal{N}_t} k^2\partial_s\kappa^2\,\mathrm{d}s+ \sum_{p=1}^2\sum_{i=1}^3 2\left( k^{pi}\right) ^3 \partial_sk^{pi}+\zeta^{pi} \left( \partial_s k^{pi}\right)^2\,
\biggr\vert_{\text{{ at    the $3$--point $O^p$}}}\nonumber\\
&\,+ 2\partial_t \sum_{p=1}^2\sum_{i=1}^3\zeta^{pi} \left( k^{pi}\right) ^2\bigr/3\,
\biggr\vert_{\text{{ at    the $3$--point $O^p$}}}+C\nonumber\\
\le&\, 
-2\int_{\mathcal{N}_t} (\partial_s^2\kappa)^2\,\mathrm{d}s +C\Vert k\Vert^2_{L^2}\Vert \partial_s\kappa\Vert^2_{L^\infty}
+ \sum_{p=1}^2\sum_{i=1}^3 2\left( k^{pi}\right) ^3 \partial_s k^{pi}+\zeta^{pi} \left(\partial_s k^{pi}\right)^2\,
\biggr\vert_{\text{{ at    the $3$--point $O^p$}}}\nonumber\\
&\,+ 2\partial_t \sum_{p=1}^2\sum_{i=1}^3\zeta^{pi}
\left(k^{pi}\right)^2\bigr/3\,
\biggr\vert_{\text{{ at    the $3$--point $O^p$}}}+C\,.\label{kappast2}
\end{align}
Using the previous estimate on $\Vert \partial_s\kappa \Vert_{L^\infty}$,
the hypothesis of uniform boundedness of $\Vert \kappa\Vert_{L^2}$ and Young inequality, we get
\begin{align*}
\Vert \kappa\Vert^2_{L^2}\Vert \partial_s\kappa\Vert^2_{L^\infty}
&\leq C+C\Vert \partial_s\kappa\Vert_{L^2}^2+C\Vert \partial_s\kappa\Vert_{L^2}\Vert \partial_s^2\kappa\Vert_{L^2}\\
&\leq C+C\Vert \partial_s\kappa\Vert_{L^2}^2+C_\varepsilon \Vert \partial_s\kappa\Vert_{L^2}^2+\varepsilon\Vert \partial_s^2\kappa\Vert_{L^2}^2\\
&=C+C_\varepsilon \Vert \partial_s\kappa\Vert_{L^2}^2+\varepsilon\Vert \partial_s^2\kappa\Vert_{L^2}^2\,,
\end{align*}
for any small value $\eps>0$ and a suitable constant $C_\eps$.\\
We deal now with the boundary term $\sum_{i=1}^3 2\left(\kappa^i\right)^3 \partial_s\kappa^i+\zeta^i \left( \partial_s\kappa^i\right)^2$.\\
Combining~\eqref{derivatatau} and~\eqref{sommakzero}
 it follows that
$\left(\partial_s\kappa+\zeta k\right)^2\sum_{i=1}^3\zeta^i=0$, hence,
$$
\sum_{i=1}^3\zeta^i \left( \partial_s\kappa^i\right)^2
=-\sum_{i=1}^3\left( \zeta^i\right)^3\left(\kappa^i\right)^2+2\left(\zeta^i\right)^2 \kappa^i\partial_s\kappa^i\,,$$
then, we can write
\begin{align}
\sum_{i=1}^3 2\left(\kappa^i\right)^3 \partial_s\kappa^i+\zeta^i \left(\partial_s\kappa^i\right)^2
=&\,\sum_{i=1}^3 \left[2\left(\kappa^i\right)^3 \partial_s\kappa^i
-\left( \zeta^i\right)^3\left(\kappa^i\right)^2-2\left(\zeta^i\right)^2 \kappa^i\partial_s\kappa^i\right]\nonumber\\
=&\,\sum_{i=1}^3 2\bigl[\left( \kappa^i\right)^3-\left( \zeta^i\right)^2 \kappa^i\bigr] \partial_s\kappa^i -\sum_{i=1}^3\left( \zeta^i\right)^3 \left( \kappa^i\right) ^2\nonumber\\
=&\,2(\partial_s\kappa+\zeta k)\sum_{i=1}^3\left[\left( \kappa^i\right)^3-\left( \zeta^i\right)^2 \kappa^i\right]
+\sum_{i=1}^3 \left[\left( \zeta^i\right)^3 \left(\kappa^i\right)^2-2\zeta^i\left( \kappa^i\right)^4\right]\,.\label{Eq1}
\end{align}
At the triple junction $\mathcal{O}^1$, where the curves $\gamma^1,\gamma^2$ and $\gamma^5$
concur, there exists $i\in\{1,2\}$ such that $\vert \kappa^i(\mathcal{O}^1)\vert\geq \frac{K}{2}$,
where $K=\max_{j\in\{1, 2, 5\}}\vert k^j(\mathcal{O}^1)\vert$.
Moreover we remind the reader that at the $3$--point $\mathcal{O}^1$, for $j\in\{1,2,5\}$, we have that
$\vert \zeta^j\vert \leq CK$. Hence at the $3$--point $\mathcal{O}^1$ it holds
\begin{align*}
2(\partial_s\kappa+\zeta k)\sum_{i=1}^3\left[\left( \kappa^i\right)^3-\left( \zeta^i\right)^2 \kappa^i\right]+\sum_{i=1}^3 &\,\left[\left( \zeta^i\right)^3 \left(\kappa^i\right)^2-2\zeta^i\left( \kappa^i\right)^4\right]\\
& \le C K^5+C\vert \partial_s\kappa^i(\mathcal{O}^1)\vert K^3
\\
&\leq C\vert \kappa^i(\mathcal{O}^1)\vert^5+C\vert \partial_s\kappa^i(\mathcal{O}^1)\vert \vert \kappa^i(\mathcal{O}^1)\vert^3
\\
&\leq C\Vert \kappa^i\Vert_{L^\infty(\gamma^i)}^5+C\Vert \partial_s\kappa^i\Vert_{L^\infty(\gamma^i)}\Vert \kappa^i\Vert_{L^\infty(\gamma^i)}^3\,.
\end{align*}
By hypothesis the length of the curves $\gamma^1$ and $\gamma^2$ is uniformly bounded away from zero, so we can estimate  $C\Vert \kappa^i\Vert_{L^\infty(\gamma^i)}^5+C\Vert \partial_s\kappa^i\Vert_{L^\infty(\gamma^i)}\Vert \kappa^i\Vert_{L^\infty(\gamma^i)}^3$ via
Gagliardo--Nirenberg interpolation inequalities. We get 
\begin{align}
\Vert \kappa^i\Vert_{L^\infty(\gamma^i)}&\leq C\Vert \partial_s^2\kappa^i\Vert^{\frac14}_{L^2(\gamma^i)}\Vert \kappa^i\Vert_{L^2(\gamma^i)}^{\frac34}
+\frac{B}{L^{\frac12}}\Vert \kappa^i\Vert_{L^2(\gamma^i)}
\leq C\Vert \partial_s^2\kappa^i\Vert_{L^2(\gamma^i)}^\frac14+C_L\label{GN3}\\
\Vert \partial_s\kappa^i\Vert_{L^\infty(\gamma^i)}&\leq C\Vert \partial_s^2\kappa^i\Vert^{\frac34}_{L^2(\gamma^i)}\Vert \kappa^i\Vert_{L^2(\gamma^i)}^{\frac14}
+\frac{B}{L^{\frac32}}\Vert \kappa^i\Vert_{L^2(\gamma^i)}
\leq C\Vert \partial_s^2\kappa^i\Vert_{L^2(\gamma^i)}^\frac34+C_L\label{GN1}\,,
\end{align} 
hence,
$$
C\Vert \kappa^i\Vert_{L^\infty(\gamma^i)}^5+C\Vert \kappa^i\Vert_{L^\infty(\gamma^i)}^3\Vert \partial_s\kappa^i\Vert_{L^\infty(\gamma^i)}
\leq
C\Vert \partial_s^2\kappa^i\Vert_{L^2(\gamma^i)}^\frac54+
C\Vert \partial_s^2\kappa^i\Vert_{L^2(\gamma^i)}^\frac32+ C_L
\leq
\varepsilon\Vert \partial_s^2\kappa^i\Vert_{L^2(\gamma^i)}^{2}+C_{L,\varepsilon}\,.
$$
Thus, finally, 
$$
2(\partial_s\kappa+\zeta \kappa)\sum_{i=1}^3\left[\left( \kappa^i\right)^3-\left( \zeta^i\right)^2 \kappa^i\right]+\sum_{i=1}^3\left[\left( \zeta^i\right)^3 \left(\kappa^i\right)^2-2\zeta^i\left( \kappa^i\right)^4\right]\leq 
\varepsilon\Vert \partial_s^2\kappa^i\Vert_{L^2(\gamma^i)}^{2}+C_{L,\varepsilon}\leq 
\varepsilon\Vert \partial_s^2\kappa\Vert_{L^2}^{2}+C_{L,\varepsilon}\,.
$$
Coming back to computation~\eqref{kappast2}, we have
\begin{align*}
\partial_t & \biggl(\int_{\mathcal{N}_t} \partial_s\kappa^2\, \mathrm{d}s-2 \sum_{p=1}^2\sum_{i=1}^3\zeta^{pi} \left(k^{pi}\right)^2
\,\big/3 \,\biggr\vert_{\text{{ at    the $3$--point $O^p$}}}\biggr)\\
&\le -2\int_{\mathcal{N}_t} \partial_s^2\kappa^2 \mathrm{d}s+C\Vert \partial_s\kappa\Vert_{L^2}^2+\varepsilon\Vert \partial_s^2\kappa\Vert_{L^2}^2+C_{L,\varepsilon}\\
&\le -2\int_{\mathcal{N}_t} \partial_s^2\kappa^2 \mathrm{d}s+C\Vert \partial_s\kappa\Vert_{L^2}^2+2\varepsilon\Vert \partial_s^2\kappa\Vert_{L^2}^2-C_{L,\eps}\Vert \kappa^i\Vert_{L^\infty(\gamma^i)}^3+C_{L,\varepsilon}\\
&\le C_{L,\varepsilon}\biggl(\int_{\mathcal{N}_t} \partial_s\kappa^2\, \mathrm{d}s-2 \sum_{p=1}^2\sum_{i=1}^3\zeta^{pi} \left(k^{pi}\right)^2\,\bigl/3
\, \biggr\vert_{\text{{ at    the $3$--point $O^p$}}}\biggr)+C_{L,\varepsilon}\,,
\end{align*}
where we chose $\varepsilon<1$.\\
By Gronwall's Lemma, it follows that $\|\partial_s\kappa\|_{L^2}^2-2 \sum_{p=1}^2\sum_{i=1}^3\zeta^{pi} \left(k^{pi}\right)^2
\,\bigl/3\, \Bigr\vert_{\text{{ at    the $3$--point $O^p$}}}$ is uniformly bounded, for $t\in[0,T)$, by a constant depending on $L$ and its value on the initial network $\mathcal{N}_0$. Then, applying Young inequality to estimate~\eqref{inkappa} of Lemma~\ref{kinfty}, there holds
$$
\|\kappa\|_{L^\infty}^3 \le   C+C\|k\|_{L^2}^{3/2} \|\partial_s\kappa\|_{L^2}^{3/2}
\le   C+C_\eps\|\kappa\|_{L^2}^{6} +\eps\|\partial_s\kappa\|_{L^2}^{2}\leq 
C_\eps+\eps \|\partial_s\kappa\|_{L^2}^{2}\,,
$$
as $\Vert \kappa\Vert_ {L^2}$ is uniformly bounded in $[0,T)$. Choosing $\eps>0$ small enough, we conclude that also $\|\partial_s\kappa\|_{L^2}$ is uniformly bounded in $[0,T)$.
\end{proof}

\section{Behavior at singular times}

The aim of this section is to improve Proposition~\ref{longtime}.
We show that the length of at least one curve must go to zero.
Moreover, if we suppose that $\mathcal{N}_t$ is a flow of regular tree--like networks,
then we prove that at the singular time the curvature remains bounded.

\subsection{Vanishing of curves}

\begin{prop}\label{regnocollapse}
Let $T\in (0,+\infty)$ and let $\mathcal{N}_t$ be a maximal solution to the motion
by curvature in $[0,T)$.
Then the inferior limit of the length of at
least one curve is zero, as $t\to T$.
\end{prop}
\begin{proof}
Assume that $T<+\infty$ is the maximal time of existence
and suppose by contradiction that
lengths of all the curves of the network are
uniformly positively bounded from below.
Then, by Proposition~\ref{longtime},
as $t\to T$ the maximum of the modulus of the curvature goes to $+\infty$. 
We perform now a parabolic rescaling (see Definition~\ref{pararesc}). at any reachable,
interior point $p_0\in\Omega$.
By Lemma~\ref{enanrem} we get that the 
limit network flow is  a static flow of either a 
straight lines with unit
multiplicity, or of a  standard triod
with unit
multiplicity.
In both cases  $\widehat{\Theta}(p_0)\leq 3/2$.
If  instead   we rescale at an end--point $P^r\in\partial\Omega$ we get
the static flow of 
 a halfline and
$\widehat{\Theta}(p_0)=1/2$.
Notice that if we parabolically rescale the flow 
$x_0\in \overline{\Omega}\setminus\mathcal{R}$, then the 
blow--up limit is the empty set, and so $\widehat{\Theta}(x_0)=0$. 
By Corollary~\ref{regcol}, we then conclude that 
the curvature is uniformly locally  bounded along the flow, around
such point $p_0$.
By the compactness of the set of reachable points ${\mathcal{R}}$, this argument clearly 
implies that the curvature of the whole $\mathcal{N}_t$ 
is uniformly bounded, as $t\to T$, which is a contradiction.
\end{proof}

We have just shown that if as $t\to T$ the length of all the curves
of the network remains uniformly bounded away from zero, then 
as $t\to T$ the curvature does not blow up and the flow
is still smooth. 

It remains to understand if as $t\to T$
either the (inferior) limit of the length of at least one curve is zero 
and the curvature explode, or we have a change of topology with bounded curvature.

\subsection{The curvature remains bounded}

We state here two lemma on the regularity of the evolving networks
needed to prove Theorem~\ref{cross}.

\begin{lem}\label{boh} 
Given a sequence of smooth curvature flows of networks 
$\mathcal{N}_t^i$ in a time interval $(t_1,t_2)$ with uniformly bounded length ratios, if in a dense subset of times $t\in(t_1,t_2)$ the networks $\mathcal{N}^i_t$ converge in a ball $B\subseteq\R^2$ in $C^1\loc$, as $i\to\infty$, to a multiplicity--one, embedded, $C^\infty$--curve $\gamma_t$ moving by curvature in $B'\supset\overline{B}$, for $t\in(t_1,t_2]$ (hence, the curvature of $\gamma_t$ is uniformly bounded), then for every $(x_0,t_0)\in B\times(t_1,t_2]$, the curvature of $\mathcal{N}_t^i$ is uniformly bounded in a neighborhood of $(x_0,t_0)$ in space--time. It follows that, for every $(x_0,t_0)\in B\times(t_1,t_2]$, we have $\mathcal{N}_t^i\to\gamma_t$ smoothly around $(x_0,t_0)$ in space--time (possibly, up to local reparametrizations of the networks $\mathcal{N}_t^i$).
\end{lem} 
\begin{proof}
See~\cite[Lemma 8.18]{ManNovPluSchu}.
\end{proof}

\begin{lem}\label{thm:locreg.3} 
Let $\mathbb{X}$ be the static flow given by a standard triod centered at the origin
and let $\mathcal{N}^i_t$ for $t\in(-1,0)$ be a sequence of smooth curvature flows of networks with uniformly bounded length ratios.
Suppose that the sequence $\mathcal{N}^i_t$  converges to ${\mathbb{X}}$ in $C^1\loc$ for almost every $t\in(-1,0)$, as $i\to\infty$.
Then the convergence is smooth on any subset of the form
$B_R(0)\times [\widetilde{t},0)$ where $R>0$ and $-1<\widetilde{t}<0$.
\end{lem}
\begin{proof}
One gets the desired result repeating the arguments
presented in~\cite[Lemma 9.1]{ManNovPluSchu}.
\end{proof}

\begin{thm}\label{cross}
Let $\mathcal{N}_0$ be a regular initial network and let 
$\mathcal{N}_t$ be a maximal solution to the motion by curvature
in the maximal time interval $[0,T)$ with initial network $\mathcal{N}_0$.
Let $p_0$ be a reachable point for the flow 
and let $\mathcal{N}^{\lambda_i}_\tau$ be any 
sequence of rescaled curvature flows around $(p_0,T)$,
as $i\to\infty$, that converges in $C^{1,\alpha}\loc \cap W^{2,2}\loc$ for almost all $\tau\in (-\infty, 0)$ and for any $\alpha \in (0,1/2)$, to a standard cross $\mathcal{N}^\infty_\tau$. 
Then,
$$
\vert \kappa(x,t)\vert\leq C<+\infty
$$
for all $t\in[0,T)$ and $x\in [0,1]$ such that $p=\gamma(x)$ is in a neighborhood of $p_0$.
\end{thm}

\begin{proof}
By the hypotheses, we can assume that the sequence of rescaled networks $\mathcal{N}^{\lambda_i}_{-1/(2+\delta)}$ converges in $W^{2,2}\loc$, as $i\to\infty$, to a standard cross (which has zero curvature), for some $\delta>0$ as small as we want.\\
By means of Lemma~\ref{thm:locreg.3} and Lemma~\ref{boh}, we can also assume
that, for $R>0$ large enough, the sequence of rescaled flows $\mathcal{N}_\tau^{\lambda_i}$ converges smoothly and uniformly to the flow $\mathcal{N}^\infty_\tau$, given by the four halflines, in $\bigl(B_{3R}(0)\setminus B_R(0)\bigr)\times[-1/2,0)$. Hence, there exists $i_0\in\NN$ such that for every $i\geq i_0$ the flow $\mathcal{N}_t$ in the annulus $B_{3R/\lambda_i}(x_0)\setminus B_{R/\lambda_i}(x_0)$ has equibounded curvature, no $3$--points and an uniform bound from below on the lengths of the four curves, for $t\in [T-\lambda_i^{-2}/(2+\delta),T)$. Setting $t_i=T-\lambda_i^{-2}/(2+\delta)$, we have then a sequence of times $t_i\to T$ such that, when $i\geq i_0$, the above conclusion holds for the flow $\mathcal{N}_t$ in the annulus $B_{3R\sqrt{2(T-t_i)}}(x_0)\setminus B_{R\sqrt{2(T-t_i)}}(x_0)$ and with $t\in[t_i,T)$, we can thus introduce four ``artificial'' moving boundary points $P^r(t)\in\mathcal{N}_t$ with $\vert P^r(t)-x_0\vert=2R\sqrt{2(T-t_i)}$, with $r\in\{1,
2, 3, 4\}$ and $t\in [t_i,T)$, such that the estimates~\eqref{endsmooth} are satisfied, that is, the hypotheses
about the end--points $P^i(t)$ of Lemmas~\ref{kinfty},~\ref{kappa2} and~\ref{kappaesse2} hold.\\
As we the sequence of networks $\mathcal{N}^{\lambda_i}_{-1/(2+\delta)}$
converges in $W^{2,2}\loc$ to a limit network with zero curvature, as $i\to\infty$, we have
$$
\lim_{i\to\infty}\Vert
\widetilde{k}\Vert_{L^2(B_{3R}(0)\cap\,\mathcal{N}^{\lambda_i}_{-1/(2+\delta)})}=0\,,\qquad
\text{ that is, }\qquad 
\int_{B_{3R}(0)\cap\,\mathcal{N}^{\lambda_i}_{-1/(2+\delta)}}\widetilde{k}^2\,\mathrm{d}\sigma\leq\varepsilon_i\,,
$$
for a sequence $\varepsilon_i\to 0$ as $i\to\infty$.
Rewriting this condition for the non--rescaled networks, we have
\begin{equation}\label{notrescaled}
\int_{B_{3R\sqrt{2(T-t_i)}}(x_0)\cap\mathcal{N}_{t_i}} k^2\,ds\leq \frac{\varepsilon_i}{\sqrt{2(T-t_i)}}\,.
\end{equation}
Applying now Lemma~\ref{kappa2} to the flow of networks $\mathcal{N}_t$ in the ball $B_{2R\sqrt{2(T-t_i)}}(x_0)$ in the time interval
  $[t_i,T)$, we have that $\Vert k\Vert_{L^2(B_{2R\sqrt{2(T-t_i)}}(x_0)\cap\mathcal{N}_t)}$ is uniformly bounded, up to
    time 
$$
T_i=t_i+\min\,\Bigl\{ T, 1\big/
8C\,\bigl(\Vert k
  \Vert^2_{L^2(B_{2R\sqrt{2(T-t_i)}}(x_0)\cap\mathcal{N}_{t_i})}+1\bigr)^2\Bigr\}\,.
$$
We want to see that actually $T_i>T$ for $i$ large enough, hence, $\Vert k
  \Vert_{L^2(B_{2R}(x_0)\cap\mathcal{N}_t)}$ is uniformly bounded for
    $t\in[t_i,T)$. If this is not true, we have
\begin{align*}
T_i=&\,t_i+\frac{1}{8C\,\bigl(\Vert k  \Vert^2_{L^2(B_{2R\sqrt{2(T-t_i)}}(x_0)\cap\mathcal{N}_{t_i})}+1\bigr)^2}\\
\geq&\,t_i+\frac{1}{8C\,\bigl(\eps_i/\sqrt{2(T-t_i)}+1\bigr)^2}\\
=&\,t_i+\frac{2(T-t_i)}{8C\,\bigl(\eps_i+\sqrt{2(T-t_i)}\,\bigr)^2}\\
=&\,T+(2(T-t_i))\biggl(\frac{2}{8C\,\bigl(\eps_i+\sqrt{2(T-t_i)}\,\bigr)^2}-1\biggr)\,,
\end{align*}
which is clearly larger than $T$, as $\eps_i\to0$, when
$i\to\infty$.

Choosing then $i_1\geq i_0$ large enough, since $\Vert
k\Vert_{L^2(B_{2R\sqrt{2(T-t_{i_1})}}(x_0)\cap\,\mathcal{N}_t)}$ is 
uniformly bounded for all times $t\in[t_{i_1},T)$ and 
the length of the four curves that connect the junctions with the ``artificial'' boundary points $P^r(t)$
are bounded below by a uniform constant, Lemma~\ref{kappaesse2}
applies, hence, thanks to Lemma~\ref{kinfty}, we have a uniform bound 
on $\Vert k\Vert_{L^\infty(B_{2R\sqrt{2(T-t_{i_1})}}(x_0)\cap\,\mathcal{N}_t)}$ for $t\in[0,T)$.
\end{proof}

\begin{proof}[Proof of Theorem~\ref{bddcurvature}]
Let $\mathcal{N}_t$ be a smooth flow in the maximal time interval $[0,T)$
of the initial network $\mathcal{N}_0$. 
Let us consider a sequence of parabolically rescaled curvature flows $\mathcal{N}^{\lambda_i}_\tau$ around $(x_0,T)$. 
By Corollary~\ref{possibiliblowup},
if we suppose that $p_0\not\in\partial\Omega\cap\mathcal{R}$, then $\mathcal{N}^\infty_\tau$ can only be the static flow given by:
\begin{itemize}
\item a straight line;
\item a standard triod;
\item a standard cross.
\end{itemize}
By White's local regularity theorem~\cite{Wh},
if the sequence of rescaled curvature flows converges to a straight line, the curvature is
uniformly bounded for $t\in[0,T)$ in a ball around the point $p_0$.
Thanks to Theorem~\ref{thm:locreg.2} and Corollary~\ref{regcol}, the same holds also in the case of the standard triod. 
In this case of a standard cross
the fact that the curvature is locally uniformly 
bounded during the flow, around the point $p_0$ is granted by Theorem~\ref{cross}.
If instead $p_0\in\partial\Omega$, the only two possibilities for $\mathcal{N}^\infty_\tau$ are the static flows given by:
\begin{itemize}
\item a halfline;
\item two concurring halflines forming an angle of $120$ degrees.
\end{itemize}
For both these two situation the thesis is obtained as in the case in which $p_0\in\Omega$
by a reflection argument (see~\cite{ManNovPluSchu}, just before Theorem~14.4).
\end{proof}

\subsection{Limit networks}

Now that we know that the curvature is uniformly bounded along the flow, we can show that the
length of the curves of the evolving network converge to a limit as $t\to T$.
Clearly, if $T$ is a singular time, then at least one of these limits must be zero.

\begin{lem}\label{limite}
Let $\mathcal{N}_t$ be a maximal solution to the motion by curvature 
such that the curvature is uniformly bounded in a time interval $[0,T)$.
Then, the lengths of the curves of the network $L^i(t)$ converge to some limit, as $t\to T$.
\end{lem}
\begin{proof}
For every curve $\gamma^i(\cdot,t)$ and for every time
$t\in[0,T)$ we have
\begin{equation}\label{levol}
\frac{dL^i(t)}{dt}=\zeta^i(1,t)-\zeta^i(0,t)-\int_{\gamma^i(\cdot,t)}k^2\,\mathrm{d}s\,.
\end{equation}
If  $\gamma^i(t,1)$  is an external vertex, then it is fixed
during the evolution and so $\zeta(t,1)=0$.
If  $\gamma^i(t,0)$ or  $\gamma^i(t,1)$ are triple junctions,
by~\eqref{zetafunzdik} we can express $\zeta(t,\cdot)$ as a linear combination
of $\kappa^{i+1}(t,\cdot)$ and $\kappa^{i-1}(t,\cdot)$.
Hence if the curvature is uniformly bounded,
then by formula~\eqref{levol}, any function
$L^i$ as a uniformly bounded derivative and the conclusion follows.
\end{proof}

\begin{lem}\label{remhot}
If {\bf{M1}} holds, there exist the limits $x_i=\lim_{t\to T}O^i(t)$, for
$i\in\{1,2,\dots,m\}$ and the set $\{x_i=\lim_{t\to T}O^i(t) ~ \vert ~ i=1,2,\dots,m\}$ is the union
of the set of the points $x$ in $\Omega$ where
$\widehat\Theta(x)>1$ with the set of the end--points of $\mathcal{N}_t$ such that
the curve getting there collapses as $t\to T$.
\end{lem}
\begin{proof}
See~\cite[Lemma 10.7]{ManNovPluSchu}. 
\end{proof}

\begin{prop}\label{bdcurvcollapse} 
Let $\Omega$ be a bounded, convex and open subset of $\mathbb{R}^2$.
Let $\mathcal{N}_0\subseteq\Omega$ be an initial regular tree--like 
network composed of $N$ curves,
with $m$ triple junctions
$\mathcal{O}^1, \ldots,\mathcal{O}^m$
 and possibly with $\ell$ external vertices 
$P^1,\ldots,P^\ell\in\partial\Omega$.
Let $\mathcal{N}_t=\bigcup_{i=1}^N\gamma^i([0,1],t)$ be a 
maximal  solution to the motion by curvature
in the maximal time interval $[0,T)$ with initial datum $\mathcal{N}_0$.
Suppose that the Multiplicity--One--Conjecture holds true.
Then as $t\to T$ the networks $\mathcal{N}_t$, up to
reparametrization proportional to arclength, converge in $C^1$ to a network $\mathcal{N}_T=\bigcup_{i=1}^M\widehat{\gamma}^i_T([0,1])$
composed of $M$ curves with positive length with $M\leq N$.
For $i\in\{1, \ldots,M\}$ the curves 
$\widehat{\gamma}^i_T([0,1])$ belong to $C^1\cap W^{2,\infty}$. Every multi--point of the $\mathcal{N}_T$ is 
\begin{itemize}
\item either a regular triple junction;
\item or a $4$--point where the four concurring curves have opposite exterior unit tangent vectors in pairs and form angles of $120/60$ degrees;
\end{itemize}
and the termini on $\partial\Omega$ are 
\begin{itemize}
\item either external vertices of order one;
\item or  $2$--points  where the
  two concurring curves form an angle of $120$
degrees.
\end{itemize}
\begin{figure}[H]
\begin{center}
\begin{tikzpicture}[rotate=90,scale=0.75]
\draw[color=black!40!white, shift={(0,-2.8)}]
(-0.05,2.65)to[out= -90,in=150, looseness=1] (0.17,2.3)
(0.17,2.3)to[out= -30,in=100, looseness=1] (-0.12,2)
(-0.12,2)to[out= -80,in=40, looseness=1] (0.15,1.7)
(0.15,1.7)to[out= -140,in=90, looseness=1.3](0,1.1)
(0,1.1)--(-.2,1.35)
(0,1.1)--(+.2,1.35);
\draw[color=black]
(-0.05,2.65)to[out= 30,in=180, looseness=1] (2,3)
(-0.05,2.65)to[out= -90,in=150, looseness=1] (0.17,2.3)
(-0.05,2.65)to[out= 150,in=-20, looseness=1] (-2,3.3)
(0.17,2.3)to[out= -30,in=100, looseness=1] (-0.12,2)
(-0.12,2)to[out= -80,in=40, looseness=1] (0.15,1.7)
(0.15,1.7)to[out= -140,in=90, looseness=1](0,1.25)
(0,1.25)to[out= -30,in=180, looseness=1] (1.9,0.7)
(0,1.25)to[out= -150,in=-15, looseness=1] (-1.9,1.2);
\draw[color=black,dashed]
(-2,3.3)to[out= 160,in=-20, looseness=1](-2.7,3.5)
 (-1.9,1.2)to[out= 165,in=-15, looseness=1](-2.7,1.3)
(1.9,0.7)to[out= 0,in=-160, looseness=1] (2.6,0.9)
(2,3)to[out= 0,in=160, looseness=1] (2.8,2.9);
\draw[color=black!40!white,shift={(0,-6)}]
(0,2.65)--(1.73,3.65)
(0,2.65)--(1.73,1.65)
(0,2.65)--(-1.73,3.65)
(0,2.65)--(-1.73,1.65);
\draw[color=black,shift={(0,-6)}]
(0,2.65)to[out= -30,in=180, looseness=1] (1.9,2)
(0,2.65)to[out= -150,in=-15, looseness=1] (-1.9,2.3)
(0,2.65)to[out= 30,in=180, looseness=1] (2.2,3.3)
(0,2.65)to[out= 150,in=-20, looseness=1] (-2.2,3.1);
\draw[color=black,dashed,shift={(0,-6)}]
(-2.2,3.1)to[out= 160,in=-20, looseness=1](-3,3.3)
 (-1.9,2.3)to[out= 165,in=-15, looseness=1](-2.7,2.4)
(1.9,2)to[out= 0,in=-160, looseness=1] (2.6,2.2)
(2.2,3.3)to[out= 0,in=160, looseness=1] (3,3.2);
\end{tikzpicture}\qquad\qquad\qquad
\begin{tikzpicture}[rotate=90,scale=0.75,shift={(15,0)}]
\draw[color=black!40!white, shift={(0,-3)}]
(-0.05,2.65)to[out= -90,in=150, looseness=1] (0.17,2.3)
(0.17,2.3)to[out= -30,in=100, looseness=1] (-0.12,2)
(-0.12,2)to[out= -80,in=40, looseness=1] (0.15,1.7)
(0.15,1.7)to[out= -140,in=90, looseness=1.3](0,1.1)
(0,1.1)--(-.2,1.35)
(0,1.1)--(+.2,1.35);
\draw[color=black]
(-0.05,2.65)to[out= 30,in=180, looseness=1] (2,3)
(-0.05,2.65)to[out= -90,in=150, looseness=1] (0.17,2.3)
(-0.05,2.65)to[out= 150,in=-20, looseness=1] (-2,3.3)
(0.17,2.3)to[out= -30,in=100, looseness=1] (-0.12,2)
(-0.12,2)to[out= -80,in=40, looseness=1] (0.15,1.7)
(0.15,1.7)to[out= -140,in=90, looseness=1](0,1.25);
\draw[color=black,dashed]
(-2,3.3)to[out= 160,in=-20, looseness=1](-2.7,3.5)
(2,3)to[out= 0,in=160, looseness=1] (2.8,2.9);
\draw[color=black!40!white]
(-3,1.7)to[out= -30,in=-180, looseness=1] (0,1.25)
(+3,1.7)to[out= -150,in=0, looseness=1] (0,1.25);
\path[font=\small]
(0,.2) node[left]{$P^r$};
\path[font=\Large]
(2.6,.5) node[left]{$\Omega$};
\draw[color=black!40!white,shift={(0,-6)}]
(0,2.65)--(1.73,3.65)
(0,2.65)--(-1.73,3.65);
\draw[color=black,shift={(0,-6)}]
(0,2.65)to[out= 30,in=180, looseness=1] (2.2,3.3)
(0,2.65)to[out= 150,in=-20, looseness=1] (-2,3.3);
\draw[color=black,dashed,shift={(0,-6)}]
(-2,3.3)to[out= 160,in=-20, looseness=1](-2.7,3.5)
(2.2,3.3)to[out=0,in=-60, looseness=1] (2.8,3.7);
\draw[color=black!40!white,shift={(0,-6)}]
(-3,3.1)to[out= -30,in=-180, looseness=1] (0,2.65)
(+3,3.1)to[out= -150,in=0, looseness=1] (0,2.65);
\fill(0,1.25) circle (2pt);
\fill(0,-3.35) circle (2pt);
\path[font=\small]
(0,-4.4) node[left]{$P^r$};
\path[font=\Large]
(2.6,-4.1) node[left]{$\Omega$};
\end{tikzpicture}
\end{center}
\begin{caption}{Collapse of a curve in the interior and at
  an end--point of $\mathcal{N}_t$.\label{Pcollapse}}
\end{caption}
\end{figure}
\end{prop}

\begin{proof} 
First of all we note that 
by Lemma~\ref{limite}
there exist the limits of the lengths of the curves $L^i(T)=\lim_{t\to T}L^i(t)$, 
for every $i\in\{1,2,\dots, N\}$. Some of these limits can be zero, we suppose that
$L^1(T),\ldots, L^M(T)$ are different from zero, with $M\leq N$.\\
We reparametrize the curves of $\mathcal{N}_t$ proportionally to their 
arclength and we relabel them as $\widehat{\gamma}^i(\cdot,t):[0,1]\to\R^2$.
We claim that for such a choice of the parametrizations, the velocities 
$\partial_t\widehat{\gamma}^i$
are uniformly bounded in space and time.
Indeed the normal velocity is the curvature, which is uniformly bounded.
Thanks to~\eqref{zetafunzdik} we can express the tangential velocity of a curve at the junctions as a linear combination
of the curvature of the other two curves, so also the tangential velocity at the junctions
is uniformly bounded. Moreover the end--points on $\partial\Omega$
are fixed during the evolution, so their tangential velocity is zero. 
Thus for every $i\in \{1,\ldots,N\}$ and $t \in [0,T)$ there exists a constant $C$
such that $\vert\widehat{\zeta}^i(t,0)\vert\leq C$, $\vert \widehat{\zeta}^i(t,1)\vert\leq C$.
It remains to deal with the tangential velocity $\widehat{\zeta}^i(t,x)$ for $x\in (0,1)$.
Because of our specific choice of the parametrizations we have
$\vert\partial_x\widehat{\gamma}^i(t,x)\vert=L(\gamma^i(t))$, hence 
$\int_0^x \vert\partial_x\widehat{\gamma}^i(x)(t,x)\vert\,\mathrm{d}x=L(\gamma^i(t))x$ with $x\in (0,1)$.
By~\eqref{levol} we have
\begin{equation}\label{primomodo}
\partial_t\int_0^x \vert\partial_x\widehat{\gamma}^i(t,x)\vert\,\mathrm{d}x=\partial_t L(\gamma^i(t))\,x
=\Bigl(\widehat{\zeta}^i(t,1)-\widehat{\zeta}^i(t,0)-\int_{\widehat{\gamma}^i(\cdot,t)}(\kappa^i)^2\,\mathrm{d}s\Bigr)\,x\,.
\end{equation}
Keeping in mind that we are working with the tangential velocity 
$$
\widehat{\zeta}^i(t,x)=\left\langle\partial_t\widehat{\gamma}^i(t,x),\tau^i(t,x)\right\rangle\,,
$$
we can write
\begin{align*}
\partial_x\widehat{\zeta}^i(t,x)&
=\left\langle\partial_x\partial_t\widehat{\gamma}^i(t,x),\tau^i(t,x)\right\rangle
+\left\langle\partial_t\widehat{\gamma}^i(t,x),\partial_x \tau^i(t,x)\right\rangle\\
&=\left\langle\partial_x\partial_t\widehat{\gamma}^i(t,x),\tau^i(t,x)\right\rangle
+\left\langle\partial_t\widehat{\gamma}^i(t,x),L(\widehat{\gamma}^i(t))\partial_s\tau^i(t,x)\right\rangle\\		
&=\left\langle\partial_x\partial_t\widehat{\gamma}^i(t,x),\tau^i(t,x)\right\rangle
+\bigl\langle\kappa^i(t,x)\nu^i(t,x)+\widehat{\zeta}^i(t,x)\tau^i(t,x),L(\widehat{\gamma}^i(t))\kappa^i(t,x)\nu^i(t,x)\bigr\rangle\\
&=\left\langle\partial_x\partial_t\widehat{\gamma}^i(t,x),\tau^i(t,x)\right\rangle
+L(\widehat{\gamma}^i(t))(\kappa^i(t,x))^2\,.
\end{align*}
This gives
\begin{equation}\label{secondomodo}
\partial_t\int_0^x \vert\partial_x\widehat{\gamma}^i(t,x)\vert\,\mathrm{d}x=
\int_0^x\left\langle\partial_t\partial_x\widehat{\gamma}^i(t,x),\tau^i(t,x)
\right\rangle\,\mathrm{d}x
=\int_0^x \partial_x\widehat{\zeta}^i(t,x)-L(\widehat{\gamma}^i(t))(\kappa^i(t,x))^2\,\mathrm{d}x\,,
\end{equation}
hence, comparing~\eqref{primomodo} and~\eqref{secondomodo}, we finally deduce that
for every $x\in (0,1)$, $t\in [0,T)$ and $i\in\{1,\ldots,N\}$, we have
\begin{equation}
\widehat{\zeta}^i(t,x)=\widehat{\zeta}^i(t,0)
+L(\widehat{\gamma}^i(t))\int_0^x(\kappa^i(t,x))^2\,\mathrm{d}x
+\Bigl(\widehat{\zeta}^i(t,1)-\widehat{\zeta}^i(t,0)-\int_{\widehat{\gamma}^i(\cdot,t)}(\kappa^i)^2\,\mathrm{d}s\Bigr)\,x
\end{equation}
and by the boundedness of $\kappa^i$, $\widehat{\zeta}^i(0)$, $\widehat{\zeta}^i(1)$
the claim follows.
So there exists $D>0$ such that
for any $x\in[0,1]$ and every pair $t,\overline{t}\in[0,T)$
$$
\vert{\widehat{\gamma}}^i(x,t)-{\widehat{\gamma}}^i(x,\overline{t})\vert\leq
\int_t^{\overline{t}}\vert{\widehat{\gamma}}_t^i(x,\xi)\vert\,d\xi\leq D\vert t-\overline{t}\vert\,.
$$
This  implies that ${\widehat{\gamma}}^i(\cdot, t):[0,1]\to\R^2$ is a Cauchy sequence in $C^0([0,1])$, hence the network $\mathcal{N}_t$ converges uniformly to a limit family of continuous curves ${\widehat{\gamma}}^i_T:[0,1]\to \R^2$, as $t\to T$.
Without loss of generality we can suppose that $(0,0)\in\Omega$.
We have
\begin{align*}
&\sup_{x\in [0,1]}\vert\widehat{\gamma}^i(t,x)\vert\leq \mathrm{diam}(\Omega)<+\infty\,,\\
&\sup_{x\in [0,1]}\vert\partial_x\widehat{\gamma}^i(t,x)\vert= L^i(t)\leq L(\mathcal{N}_0)<+\infty\,,\\
&\sup_{x\in [0,1]}\vert\partial_x^2\widehat{\gamma}^i(t,x)\vert=\sup_{x\in [0,1]} L^i(t)\vert \kappa^i(t,x)\vert\leq C L(\mathcal{N}_0)<+\infty\,,
\end{align*}
Hence, the family of curves $\widehat{\gamma}^i(t,x)$ composing the networks $\mathcal{N}_t$, converges in $C^1$, as $t\to T$, to the family $\widehat{\gamma}^i_T$ composing $\widehat{\mathcal{N}}_T$ and  by the uniform bound on the curvature, all the curves $\widehat{\gamma}^i_T$ belong to $W^{2,\infty}$. In particular, for $i\in\{1, \ldots,M\}$
the curves $\widehat{\gamma}^i_T$ are parametrized proportionally to arclength
and for $i\in\{M+1, \ldots,N\}$ we have that $\widehat{\gamma}^i_T$ are constant maps.

We deal now with the convergence of the unit tangent vectors. We observe that if we denote with $s$ the arclength parameter, we have
\begin{equation}\label{taudeg}
\biggl\vert\frac{\partial \widehat{\tau}^i(x,t)}{\partial x}\biggr\vert=
\biggl\vert\frac{\partial \tau^i(s,t)}{\partial s}\biggr\vert L^i(t)=
\vert \kappa^i(s,t)\vert L^i(t)\leq CL^i(t)\leq C^2\,,
\end{equation}
for some constant $C$, hence, for every sequence of times $t_n\to T$ have a -- not relabeled --
subsequence such that the maps $\widehat{\tau}^i(\cdot,t_n)$
converge uniformly to some maps $\widehat{\tau}_T^i$.
For any curve $\widehat{\gamma}^i(t,x)$ and for any $x,y\in [0,1]$ and $t\in[0,T)$ we also have
\begin{align*}
&\sup_{x,y\in [0,1]}\vert \tau^i(t,x)-\tau^i(t,y)\vert
=\frac{1}{L^i(t)}\sup_{x,y\in [0,1]}\vert \partial_x\gamma^i(t,x)-\partial_x\gamma^i(t,y)\vert\\
&\leq \frac{1}{L^i(t)}\sup_{x,y\in [0,1]}\int_x^y\vert\partial_x^2\gamma^i(t)\vert\,\mathrm{d}t
= \frac{1}{L^i(t)}\sup_{x,y\in [0,1]}\int_x^y\vert\kappa^i(t)(L^i(t))^2\,\mathrm{d}t
\leq \vert x-y\vert C L^i(t)\,.
\end{align*}

Consider $i\in\{1,\ldots,M\}$, that is $L^i(t)$ does
not go to zero and 
the curve $\widehat{\gamma}^i_T$ is a regular curve. 
It is easy to see that the limit maps $\widehat{\tau}_T^i$ must
coincide with the unit tangent vector field to
the curve $\widehat{\gamma}^i_T$, hence, the full sequence
$\widehat{\tau}^i(\cdot,t)$ converges to $\widehat{\tau}^i_T$.

If instead $i\in\{M+1,\ldots,N\}$ and $L^i(T)=0$, 
then the limit as $t\to T$ of 
$$
\sup_{x,y\in [0,1]}\vert \tau^i(t,x)-\tau^i(t,y)\vert
$$
is zero. Then, for every $x\in [0,1]$
 the maps $\widehat{\tau}^i(t_n,x)$
converge to a constant unit vector $\widehat{\tau}_T^i$.

Call $\gamma$ a curve whose end--points are two triple junctions $\mathcal{O}^1,\mathcal{O}^2$
and  for $i\in\{1,2\}$ call $\varphi^i,\psi^i$ the other two curves concurring at $\mathcal{O}^i$.
Suppose that the length of  $\gamma$ goes to zero and the lengths of 
$\varphi^i,\psi^i$ remain strictly positive, as $t\to T$.
Thanks to the information we have collected and the angle conditions at the triple junctions
in the limit $\mathcal{N}_T$, the two triple junctions $\mathcal{O}^1,\mathcal{O}^2$
coincide and the four curve  $\varphi^1,\varphi^2,\psi^1,\psi^2$ concur at the junctions, 
forming a four point where the unit tangent vectors form angles of $120$ and $60$ degrees. 

Now call $\gamma$ a curve with one internal and one external vertex $P\in\partial\Omega$
and $\varphi,\psi$ of the other two curves that concur 
at the common triple junctions.
Suppose that the length of  $\gamma$ goes to zero and the lengths of 
$\varphi,\psi$ remains strictly positive 
as $t\to T$.
If we symmetrize the initial network around $P$ (central symmetry with center $P$)
and then we let evolve this network, then 
we have a flow or regular networks 
and the symmetry is preserved during the flow.
So, as $t\to T$, the triple junction $\mathcal{O}$ and it symmetric one $\mathcal{O}'$
coalesce onto $P$. This is exactly the phenomena described in the previous case.
Looking only at the original curves, at $T$ we have that two curves arrived at $P$
and their unit tangent vectors form $120$ degrees.

To conclude the proof we must exclude that the length of two adjacent curves 
goes to zero as $t\to T$.
Recalling now Lemma~\ref{remhot}, it is enough to analyze the flow locally around every point $p_0$ which belongs to the set of the limits of the triple junctions $\{O^j(t)\}$, as $t\to T$. If the point $p_0$ is the limit point of a single triple junction $O^j(t)$, clearly locally around $p_0$ no curve is vanishing. 
Suppose that the curve $\gamma^i(\cdot, t)$ (at least) collapses with its end--points going to $p_0$. We perform a parabolic rescaling (see Definition~\ref{pararesc}). around $(p_0,T)$. 
The self--similar shrinking limit flows which we obtain in this way must be among the ones in Corollary~\ref{possibiliblowup}: the static flow of a line, of a standard triod, or of a standard cross. The first two cases are clearly excluded, since it would hold $\widehat{\Theta}(x_0)\leq 3/2$, then arguing as in Lemma~\ref{lemmatree} (see Figure~\ref{fig8}, in particular), a collapse of curves is not possible, otherwise $\widehat{\Theta}(x_0)\geq2$. Hence, the only possibility is a standard cross (which has a core composed only of a single collapsed curve).
This actually implies that at $p_0$ there are no collapsing curves other than $\gamma^i(\cdot,t)$ and only its end--points (among the triple junctions) are converging to $p_0$. As 
\begin{equation}\label{Oconv2}
\vert O^j(t_2)-O^j(t_1)\vert=\biggl\vert\int_{t_1}^{t_2}v^j(\xi)\,d\xi\,\biggl\vert\leq
\int_{t_1}^{t_2}\vert v^j(\xi)\vert\,d\xi\leq D\vert t_2-t_1\vert\,,
\end{equation}
for every triple junction $O^j(t)$ converging to $p_0$, for every $t_1,t_2\in[0,T)$, hence
$$
\vert O^j(t)-x_0\vert\leq D\vert T-t\vert
$$
for every $t\in[0,T)$, we have that its image $\widetilde{O}^j(\tau)$, after performing the parabolic  rescaling procedure, satisfies
$$
\vert\widetilde{O}^j(\tau)\vert=\frac{\vert O^j(t(\tau)-x_0\vert}{\sqrt{2(T-t(\tau))}}\leq\frac{D\vert T-t(\tau)\vert}{\sqrt{2(T-t(\tau))}}=D\sqrt{(T-t(\tau))/2}\,,
$$
which tends to zero, as $\tau\to+\infty$, in particular, all the triple junctions converging to $p_0$ cannot ``disappear'' in the limit (going to infinity). As the core of the standard cross is a single line, the above claim follows and the collapsing curve $\gamma^i(\cdot,t)$ is ``isolated'', as in the statement.
\end{proof}

\begin{proof}[Proof of Theorem~\ref{main}]
We only have to show that the curves of $\mathcal{N}_T$ are actually $C^2$. By means of Lemma~\ref{kappaesse2}, $\Vert \partial_s\kappa\Vert_{L^2}$ is locally uniformly bounded on $[0,T)$, which implies that the convergence of the non--collapsing curves of $\mathcal{N}_t$ to $\mathcal{N}_T$, as $t\to T$, is actually in $C^2\loc$ and we are done.
\end{proof}

\section{An example of a Type-$\mathbf 0$ singularity}\label{extypezero}

We have proved that during the flow of regular tree like network the curvature
remains bounded, even if the length of a curve goes to zero.
We finally show that our main theorem is sharp in the sense that the coalesce of
two triple junctions during the flow might actually happen.

\medskip

Consider an initial regular smooth network $\mathcal{N}_0$ composed of five curves, which is centrally symmetric, in the convex domain $\Omega$ (also centrally symmetric).
Suppose that four curves have one end--point fixed on $\partial\Omega$
and there are two triple junctions (see Figure~\ref{type0fig}).

\begin{figure}[H]
\begin{center}
\begin{tikzpicture}[scale=0.30]
\draw[color=black,scale=3.3,domain=-3.141: 3.141,
smooth,variable=\t,shift={(0,0)},rotate=0]plot({5*sin(\t r)},
{1.5*cos(\t r)});
\draw[color=black]
(0,0)to[out= 90,in=-90, looseness=1](0,2)
(0,2)to[out=150,in=-50, looseness=1] (-12,3.46)
(0,2)to[out= 30,in=-160, looseness=1] (12,3.46);
\draw[color=black, rotate=180]
(0,0)to[out= 90,in=-90, looseness=1](0,2)
(0,2)to[out=150,in=-50, looseness=1] (-12,3.46)
(0,2)to[out= 30,in=-160, looseness=1] (12,3.46);
\draw[color=black!30!white]
(-10,0)to[out= 0,in=180, looseness=1](10,0)
(-10,0)to[out= 120,in=-60, looseness=1] (-12,3.46)
(-10,0)to[out= -120,in=60, looseness=1] (-12,-3.46)
(10,0)to[out= 60,in=-120, looseness=1] (12,3.46)
(10,0)to[out=-60,in=120, looseness=1] (12,-3.46);
\path[font=]
(15,-2.8) node[below] {\Large{$\Omega$}}
(-1,4.5) node[below] {\large{$\mathcal{N}_0$}}
(12.2,2) node[below] {\large{${\mathbb M}$}};
\end{tikzpicture}
\end{center}
\begin{caption}{The networks $\mathcal{N}_0$ and ${\mathbb M}$.\label{type0fig}}
\end{caption}
\end{figure}
\noindent 
If the triangle given by the two end--points in the upper half plane and the origin
has the angle at the origin greater than $120$ degrees, then  there exists only one 
minimal network ${\mathbb M}$ connecting the four end--points of $\mathcal{N}_0$ on the boundary of $\Omega$. 
During the smooth curvature flow $\mathcal{N}_t$ of $\mathcal{N}_0$ (that maintains the central symmetry)  the curvature is bounded and either a singularity develops or the flow $\mathcal{N}_t$ is smooth for every positive time. In this second case, by Proposition~13.5 in~\cite{ManNovPluSchu} (a standard energy--based argument), as $t\to+\infty$, the network $\mathcal{N}_t$ converges in $C^1$ to ${\mathbb M}$, which is not possible because of their different geometric ``structures''. Hence, at some time $T<+\infty$ a Type-0 singularity must develop and the only possibility is the collapse of the ``central'' curve of $\mathcal{N}_t$, by its symmetry.

\bibliographystyle{amsplain}
\bibliography{collasso2tripunti}
\end{document}